\documentclass[11pt,reqno]{amsart}
\usepackage{amsbsy,amsfonts,amsmath,amssymb,amscd,amsthm}
\usepackage{mathrsfs,float}
\usepackage[mathcal]{euscript}
\usepackage{enumitem,graphicx,epstopdf}
\usepackage[abs]{overpic}
\usepackage[a4paper,text={6.1in,8in},centering]{geometry}
\usepackage{pxfonts,epstopdf}
\usepackage[colorlinks=true,linkcolor=blue,citecolor=green]{hyperref}
\usepackage{srcltx}
\usepackage[margin=20pt,font=small,labelfont=bf,textfont=it]{caption}
\usepackage{graphicx}
\usepackage{subcaption}
\usepackage{pinlabel}
\usepackage[latin1]{inputenc}

\hypersetup{linktocpage}

\small\normalsize
\theoremstyle{plain}
\newtheorem{mainthm}{Theorem}

\newtheorem{thm}{Theorem}[section]
\newtheorem{lem}[thm]{Lemma}

\newtheorem{prop}[thm]{Proposition}
\newtheorem*{thm*}{Theorem}

\newtheorem{defi}[thm]{Definition}
\newtheorem{rem}[thm]{Remark}

\theoremstyle{definition}

\newtheorem{claim}{Claim}

\newcommand{\eqdef}{\stackrel{\scriptscriptstyle\rm def}{=}}

\setlist[enumerate,1]{label=(\arabic*)}
\setlist[enumerate,2]{label=(\alph*)}

\begin{document}

\title[Berger domains]{Berger domains and Kolmogorov typicality of infinitely many invariant circles}
\author[Barrientos]{Pablo G. Barrientos}
\address{\centerline{Instituto de Matem\'atica e Estat\'istica, UFF}
    \centerline{Rua M\'ario Santos Braga s/n - Campus Valonguinhos, Niter\'oi,  Brazil}}
\email{pgbarrientos@id.uff.br}
\author[Raibekas]{Artem Raibekas~\vspace{0.30cm}}
\address{\centerline{Instituto de Matem\'atica e Estat\'istica, UFF}
   \centerline{Rua M\'ario Santos Braga s/n - Campus Valonguinhos, Niter\'oi,  Brazil}}
\email{artemr@id.uff.br}

\begin{abstract}
Using the novel notion of parablender, P.~Berger proved that the
existence of finitely many attractors is not Kolmogorov typical in
parametric families of diffeomorphisms. Here, motivated by the
concept of Newhouse domains we define Berger domains for families
of diffeomorphisms. As an application, we show that the
coexistence of infinitely many attrac\-ting invariant smooth
circles is Kolmogorov typical in certain non-sectionally
dissipative Berger domains of parametric families in dimension
three or greater.
\end{abstract}

\maketitle \thispagestyle{empty}
\section{Introduction}
Many dynamical properties, such as hyperbolicity, are robust in
$C^r$-topology of diffeomorphisms. That is, the property holds under any appropriate small
perturbation of the dynamical system.  However, many others
interesting phenomena, non-hyperbolic strange attractors for
instance, are not stable in that sense. Hence, the question that
arises is whether such dynamical properties could be survive if not for
all perturbations but, at least, for most. For one-dimensional
dynamics the Malliavin-Shavgulidze measure has been recently
proposed as a good analogy to the Lebesgue measure in order to
quantify this abundance in a probabilistic
sense~\cite{Triestino}. However, in higher dimensions, it is not
known how to introduce a good notion of a measure in the space of
dynamical systems. Kolmogorov in his plenary talk ending the ICM 1954
proposed to consider finite dimensional parametric families taking into
account the Lebesgue measure in the parameter space
(see~\cite{hunt2010prevalence}). A parametric family $(f_a)_a$
exhibits \emph{persistently} a property $\mathscr{P}$ if
it is observed for $f_a$ in a set of parameter values
$a$ with positive Lebesgue measure. Furthermore, the property
$\mathscr{P}$ is called \emph{typical} (in the sense of
Kolmogorov) if there is a Baire (local) generic set of parametric
families exhibiting the property $\mathscr{P}$
persistently with full Lebesgue measure. In this direction, a
milestone in recent history of dynamical systems has been the
paper of Berger~\cite{Ber16} (see also~\cite{Ber17}) where it was
proven that the coexistence of infinitely many periodic sinks is
Kolmogorov typical in parametric families of endomorphisms in
dimension two and diffeomorphisms in higher dimensions.
The work of Berger extends, in a measurable sense according to
Kolmogorov, the important results in the 70's due to
Newhouse~\cite{New74,New79} (see
also~\cite{robinson1983bifurcation,PT93,PV94,GST93a}) on the local
genericity of the coexistence of infinitely many hyperbolic
attractors (sinks) in $C^r$-topology. This celebrated result was coined as
\emph{Newhouse phenomena}. Mimicking this terminology we will refer
to the Kolmorov typical coexistence
of infinitely many attractors as \emph{Berger phenomena}.

Newhouse phenomena has been showen to occur in open sets of
diffeomorphisms having a dense subset of systems displaying
homoclinic tangencies associated with saddle periodic points. Such
an open set of dynamical systems is called a \emph{Newhouse
domain}. In certain cases, these open sets are also the support of
many other interesting phenomena such as the coexistence of
infinitely many attracting invariant circles~\cite{GST08} and
infinitely many strange
attractors~\cite{colli1998infinitely,leal2008high}, or wandering
domains~\cite{KS17} among others. Berger phenomena also occurs
with respect to some open set but now in the topology of
parametric families. Namely, in open sets where the families
having persistent homoclinic tangencies are dense. As before,
mimicking the terminology, we will refer to these open sets of
parametric families as \emph{Berger domains}. In the original
paper of Berger~\cite{Ber16,Ber17}, these open sets were
implicitly constructed for sectional dissipative dynamics. In this
paper, we will introduce formally the notion of a Berger domain
and construct new examples, not necessarily for sectional
dissipative dynamics. As an application, we will prove Berger
phenomena for a certain type of non-sectional dissipative Berger
domains and obtain that the coexistence of infinitely many
attracting invariant circles is also Kolmogorov typical.

\subsection{Degenerate unfoldings}

A $C^r$-diffeomorphism $f$ of a manifold $\mathcal{M}$ has a
\emph{homoclinic tangency} if there is a pair of points ${P}$ and
$Q$, in the same transitive hyperbolic set, so that the unstable
invariant manifold of ${P}$ and the stable invariant manifold of
$Q$ have a non-transverse intersection at a point $Y$. The
tangency is said to be of codimension $c>0$ if
\begin{equation*}
c= c_Y(W^u(P),W^s(Q)) \eqdef \dim M - \dim (T_YW^u(P)+T_YW^s(Q)).
\end{equation*}
This number measures how far from being transverse is the intersection between the invariant
manifolds at $Y$. Since the codimension
of $W^u(P)$ coincides with the dimension of $W^s(Q)$ we have, in
this case, that the codimension $c$ at $Y$ coincides with $\dim T_Y
W^u(P)\cap T_Y W^s(Q)$.
A homoclinic tangency can be unfolded by considering a
$k$-parameter family in the $C^{d,r}$-topology with $1\leq d\leq
r$. That is, a  $C^d$-family $(f_a)_a$ of $C^r$-diffeomorphisms
pa\-ra\-me\-te\-ri\-zed by $a\in \mathbb{I}^k$ with $f_{a_0}=f$
where $\mathbb{I}=[-1,1]$ and $k\geq 1$ (see~\S\ref{sec:topology}
for a more precise definition).
The unfolding of {a
tangency $Y$ of codimension $c$} is said to be
\emph{$C^d$-{degenerate}} at $a=a_0$ if there are points $p_a \in
W^u(P_a)$, $q_a \in W^s(Q_a)$ and $c$-dimensional subspaces $E_a$,
$F_a$ of $T_{p_a}W^u(P_a)$ and $T_{q_a}W^s(Q_a)$ respectively
such that
$$
 d(p_a,q_a)=o(\|a-a_0\|^{d}) \ \ \text{and} \ \
d(E_a,F_a)={o(\|a-a_0\|^{d})} \ \ \text{at $a=a_0$}.
$$
Here $P_a$ and $Q_a$ are the continuations of $P_{a_0}=P$ and
$Q_{a_0}=Q$ for $f_a$. Also $p_{a_0}=q_{a_0}=Y$ and $(p_a,E_a)$,
$(q_a,F_a)$ vary $C^d$-continuously with respect to the parameter
$a\in \mathbb{I}^k$. Observe that in this case it is necessary to
assume that $d<r$ because the above definition involves the
dynamics of the family $(f_a)_a$ in the tangent bundle (in fact,
in certain Grassmannian bundles). In~\cite{Ber16}, the notion of
$C^d$-degenerate unfoldings of homoclinic tangencies were
introduced for short under the name of
\emph{$C^d$-paratangencies}.

\subsection{Berger domains}
\label{BergerDomain} {{
Let us remind the reader the notion of Newhouse domains.
Following~\cite{BD12}, we say that a $C^r$-open set $\mathcal{N}$
of diffeomorphisms is a \emph{$C^r$-Newhouse domain} (of
tangencies of codimension $c>0$) if there exists a dense set
$\mathcal{D}$ in $\mathcal{N}$ such that every $g\in \mathcal{D}$
has a homoclinic tangency (of codimension $c>0$) associated with
some hyperbolic periodic saddle. A $C^r$-Newhouse domain
$\mathcal{N}$ ($r\geq 1$) of homoclinic tangencies (of codimension
one) associated with sectional dissipative periodic points gives
rise to the $C^r$-Newhouse phenomenon. Namely, there exists a
residual subset $\mathcal{R}$ of $\mathcal{N}$ where every $g\in
\mathcal{R}$ has infinitely many hyperbolic
periodic attractors
\cite{New74,New79,PT93,GST93a,PV94,Ro95,GST08}. As Berger showed
in~\cite{Ber16}, open sets of families displaying degenerate
unfoldings play the same role for parametric families as Newhouse
domains do for the case free of parameters.
For this reason
mimicking the above terminology,
one could say that:
\begin{enumerate}[label={}, rightmargin=2em, leftmargin=2em]
\item \it An open set $\mathscr{U}$ of $k$-parameter $C^d$-families of
$C^r$-diffeomorphisms is called a \emph{$C^{d,r}$-Berger domain}
of paratangencies (of codimension $c>0$) if  the following holds.
There exists a dense set $\mathscr{D}\subset \mathbb{I}^k\times
\mathscr{U}$ such that for any $(a_0, f)\in \mathscr{D}$, the
family $f=(f_a)_a$ displays a $C^d$-degenerate unfolding at
$a=a_0$ of a homoclinic tangency (of codimension $c>0$) associated
with a hyperbolic periodic~saddle.
\end{enumerate}
For codimension $c=1$, this definition appears implicitly
in~\cite{Ber16} where
it is proven that the
coexistence of infinitely many hyperbolic periodic attractors is
Kolmogorov typical.
Actually, by modifying the
initial construction Berger showed a stronger result
in~\cite{Ber16,Ber17} that we will refer to as
\emph{$C^{d,r}$-Berger phenomena}: the existence of a residual set
in a $C^{d,r}$-open set of parametric families where each family
has
infinitely many sinks at
\emph{any} parameter.
The
following
stronger version of the above tentative
definition allowed Berger to prove such a result:
\begin{defi} \label{def:Berger-domain}
An open set $\mathscr{U}$ of $k$-parameter $C^d$-families of
$C^r$-diffeomorphisms is called \emph{$C^{d,r}$-Berger domain} of
persistent homoclinic tangencies
(of codimension $c>0$) if there exists a dense subset
$\mathscr{D}$ of $\mathscr{U}$ such that for any $f=(f_a)_a \in
\mathscr{D}$ there is a covering of $\mathbb{I}^k$ by open balls
$J_i$
having the following property:
there is a continuation of a
saddle periodic point $Q_a$ having a homoclinic tangency $Y_{a}$
(of codimension $c>0$) which depends $C^d$-continuously on the
parameter $a\in J_i$.
\end{defi}

Observe that the first tentative definition above requires $d<r$
because of the notion of the $C^d$-paratangency. However,
definition~\ref{def:Berger-domain} admits
$d\leq r$
since it deals with the
notion of a $C^d$-persistent homoclinic tangency.
The following result shows the existence of Berger
domains of large codimension for families of diffeomorphisms:

\begin{mainthm}
\label{thmD} Any manifold of dimension $m> c^2+c$ admits an open
set $\mathscr{U}$ of $k$-parameter $C^d$-families of
$C^r$-diffeomorphisms with {$0<d< r-1$}, so that  $\mathscr{U}$ is
a \emph{$C^{d,r}$-Berger domain} of persistent homoclinic
tangencies of codimension $c>0$.
\end{mainthm}
}

The proof of Theorem~\ref{thmD} is based on the notion of a
$C^d$-degenerate unfolding of homoclinic tangencies and previous
results from~\cite{BR20}. For this reason, we have only been able
to show the existence of $C^{d,r}$-Berger  for families of
diffeomorphisms with $d<r-1$ and manifolds of dimension $m\geq 3$.
Recall that, in the case of codimension $c=1$, Berger, in his
original papers~\cite{Ber16} and~\cite{Ber17}, constructed this
kind of open sets for $C^d$-families of $C^r$-endomorphisms in any
surface with $1\leq d\leq r$. Afterwards this construction is lifted
to $C^d$-families of sectionally dissipative
$C^r$-diffeomorphims in manifolds of dimension $m\geq 3$. It is
unknown if Berger domains exist for families of diffeomorphisms
in dimension $m=2$.

The persistent homoclinic tangencies obtained in the above theorem
can be associated with a finite collection of saddle periodic
points, $Q_{ja}$, having unstable index $c$ and
the same type of multipliers. For instance, we can take these
points to be sectionally dissipative (as in the original
construction of Berger) but also these saddles can be of
type~$(1,1)$, $(2,1)$, $(1,2)$ or $(2,2)$ according to the
nomenclature introduced in~\cite{GST08}.
We remark that in the codimension one case we may assume the homoclinic tangencies are simple\footnote{The tangency is called simple if it is
quadratic, of codimension one and in the case the ambient manifold has dimension
$m>3$, any extended unstable manifold is transverse to the leaf of
the strong stable foliation which passes through the tangency
point. Observe that these conditions are generic.}, also in the sense of ~\cite{GST08}.

%

\subsection{Berger phenomena:}
The $C^{d,r}$-Berger phenomena was shown in~\cite{Ber16,Ber17} for
sectionally dissipative families in dimension
$m\geq 3$. We will obtain similar results for
families that are not sectionally dissipative
by working with a
$C^{d,r}$-Berger domain
$\mathscr{U}$ of type $(2,1)$ with unstable index one. That is,
the persistent homoclinic tangencies are simple and associated with
hyperbolic periodic points having multipliers
$\lambda_1,\dots,\lambda_{m-1}$ and $\gamma$ satisfying
\begin{equation} \label{eq1}
   |\lambda_j|<|\lambda|<1<|\gamma| \ \ \ \text{and \ \ \ $|\lambda^2\gamma|<1<|\lambda\gamma|$} \quad \text{for $j\not= 1,2$ where
   $\lambda_{1,2}=\lambda e^{\pm i \varphi}$ with $\varphi\not=0,\pi$.}
\end{equation}
In the following result we obtain Berger phenomena with respect to
attracting invariant circles and hyperbolic sinks
for these new types of
Berger domains.

\begin{mainthm} \label{thm-thmA} Let $\mathscr{U}$ be a
a $C^{d,r}$-Berger domain whose persistent homoclinic tangencies
are simples and associated with hyperbolic periodic points having multipliers
satisfying~\eqref{eq1}.  Then there exists a residual set
$\mathcal{R}\subset \mathscr{U}$ such that for every family
$f=(f_a)_a \in \mathcal{R}$ and every $a\in \mathbb{I}^k$, the
diffeomorphism $f_a$ has simultaneously
\begin{enumerate}
\item[-] infinitely many normally hyperbolic attracting invariant circles and
\item[-] infinitely many hyperbolic periodic sinks
\end{enumerate}
\end{mainthm}



\subsection{Topology of families of diffeomorphisms}
\label{sec:topology} Set $\mathbb{I}=[-1,1]$. Given $0<d \leq
r\leq \infty$, $k\geq 1$ and a compact manifolds $\mathcal{M}$ we
denote by $C^{d,r}(\mathbb{I}^k,\mathcal{M})$ the space of
$k$-parameter $C^d$-families $f=(f_a)_a$ of $C^r$-diffeomorphisms
$f_a$ of $\mathcal{M}$ parameterized by $a\in\mathbb{I}^k$ such
that
\begin{equation*}
  \partial^i_a \partial^j_x f_a(x) \ \ \text{exists continuously
  for all $0\leq i\leq d$, \ \   $0\leq i+j\leq r$  \ \  and \ \
   $(a,x)\in \mathbb{I}^k\times \mathcal{M}.$}
\end{equation*}
We endow this space with the topology given by the $C^{d,r}$-norm
$$
\|f\|_{{C}^{d,r}}=\max\{\sup \|\partial^i_a\partial_x^j f_a (x):
\, 0\leq i \leq d, \  0\leq i + j \leq r\} \quad \text{where \
$f=(f_a)_a \in {C}^{d,r}(\mathbb{I}^k,\mathcal{M})$.}
$$

\subsection{Structure of the paper}
Section~\S\ref{sec:berger-domain} contains the proof of
Theorem~\ref{thmD}. Independently
in section~\S\ref{sec:thm-thmA} we prove Theorem~\ref{thm-thmA}.
Actually, the proof of Theorem~\ref{thm-thmA} only requires
Definition~\ref{def:Berger-domain}.

\section{Berger domains: Proof of Theorem~\ref{thmD}}
\label{sec:berger-domain} In this section we will prove the
existence of $C^{d,r}$-Berger domains
of
codimension $u>0$ for families of diffeomorphisms with $d<r-1$ and
manifolds of dimension $m> u^2+u \geq 2$ (see \S\ref{BergerDomain}
and Definition~\ref{def:Berger-domain}).

Now we will introduce the family that will be ``the organizing
center'' of Berger domains. To do this, we need some results
from~\cite{BR20}. In~\cite[Thm.~B]{BR20} we construct an open set
$\mathscr{U}\subset C^{d,r}(\mathbb{I}^k,\mathcal{M})$
for $0<d<r-1$  and $\dim \mathcal{M} \geq 3$ where any family
$f=(f_a)_a\in \mathscr{U}$ has a $C^d$-degenerate unfolding of a
homoclinic tangency of codimension $u$ at $a=0$ (actually at any
parameter $a_0\in \mathbb{I}^k$). The construction of this open
set is local and only requires two ingredients: a family of
blenders (a certain type of a hyperbolic basic set)
 $\Gamma=(\Gamma_a)_a$ and a family of folding manifolds
$(S_a)_a$ (a certain type of manifold that folds along some
direction). We refer to~\cite{BR20} for a precise definition of
these objets.  To be more specific, the main result could be
stated as follows:

\begin{thm}[{\cite[Thm.~7.5, Rem.~7.6]{BR20}}] \label{thm-BR20}
For any $0<d<r-1$ and $k\geq 1$, there exists a $C^d$-family
$\Phi=(\Phi_a)_a$ of locally defined $C^r$-diffeomorphisms of
$\mathcal{M}$ of dimension $m>u+u^2$ having a family of
$cs$-blenders $\Gamma=(\Gamma_a)_a$ with unstable dimension $u\geq
1$
and a family of folding manifolds $\mathcal{S}=(\mathcal{S}_a)_a$
of dimension $m-u$ satisfying the following:

For any $a_0\in \mathbb{I}^k$, any family $g=(g_a)_a$ close enough
to $f$ in the $C^{d,r}$-topology and any $C^{d,r}$-perturbation
$\mathcal{L}=(\mathcal{L}_a)_a$ of $\mathcal{S}$ there exists
$z=(z_a)_a\in C^d(\mathbb{I}^k,\mathcal{M})$ such that
 {
\begin{enumerate}
\item $z_a\in\Gamma_{g,a}$, where $\Gamma_{g,a}$ denotes the
continuation for $g_a$ of the blender $\Gamma_a$,
\item the family of local unstable manifolds
$\mathcal{W}=(W^u_{loc}(z_a;g_a))_a$ and $\mathcal{L}$ have a
tangency
 of dimension $u$ at $a=a_0$ which unfolds $C^d$-degenerately.
\end{enumerate}
}
%
\end{thm}

Let us consider the family $\Phi=(\Phi_a)_a$  given in the above
theorem.  Assume in addition the next hypothesis:
\begin{enumerate}[label=(H\arabic*)]
\item \label{H1}
$\Phi_a$ has a equidimensional cycle between
saddle periodic points $P_a$ and $Q_a$,
\item \label{H2} $P_a$ belongs to $\Gamma_a$ and the folding manifold $S_a$
is contained in $W^s(Q_a,\Phi_a)$.
\end{enumerate}
Theorem~\ref{thm-BR20} implies that the family $\Phi=(\Phi_a)_a$
under the above assumptions~\ref{H1} and~\ref{H2} defines a
$C^{d,r}$-open set $\mathscr{U}=\mathscr{U}(\Phi)$ of
$k$-parameter $C^d$-families of $C^r$-diffeomorphisms
such that any $g=(g_a)_a \in \mathscr{U}$ is a $C^d$-degenerate
unfolding at any parameter $a=a_0$ of a tangency of dimension $u$.
The tangency is between $W^s(Q_{a_0},g_{a_0})$ and the local
unstable manifold of some point in the blender $\Gamma_{a_0}$ of
$g_{a_0}$.
Since the codimension of $W^s(Q_{a_0},g_{a_0})$ and the
dimension of the local unstable manifolds of $\Gamma_{a_0}$
coincide, the tangency also has codimension $u$. We will prove
that the open set $\mathscr{U}$ is a $C^{d,r}$-Berger domain. To
do this, we will first need the following result,
see~\cite[Lemma~3.7]{Ber16} and~\cite[Lemma~3.2]{Ber17}.


\begin{prop}[Parametrized Inclination Lemma] \label{thm-Parametrized-Inclination-Lemma}
Let $g=(g_a)_a$ be a $C^{d,r}$-family of diffeomorphisms having a
family $K=(K_a)_a$ of transitive hyperbolic sets $K_a$ with
unstable dimension $d_u$. Let $C_a$ be a $C^r$-submanifold of
dimension $d_u$ that intersects transversally a local stable
manifold $W^s_{loc}(x_a,g_a)$ with $x_a\in K_a$ at a point $z_a$
which we assume depends $C^d$-continuously on $a  \in
\mathbb{I}^k$. Then, for any $P_a\in K_a$
there exists a $d_u$-dimensional disc $D_a \subset C_a$
containing $z_a$ such that the family of discs $D=(g^n_a(D_a))_a$
is $C^{d,r}$-close to $W=(W^u_{loc}(P_a,g_a))_a$, for $n$ sufficiently large.
\end{prop}



Using the parameterized inclination lemma, the following
proposition proves that the above open set $\mathscr{U}$ is a
Berger domain according to the first tentative (weaker) definition given
in~\S\ref{BergerDomain}.

\begin{prop} \label{prop-dense-of-paratangencies}
For any $a_0\in \mathbb{I}^k$ and $g\in \mathscr{U}$, there is
$C^{d,r}$-arbitrarily close to $g$ a family $f=(f_a)_a$ such that
$f_{a}=g_a$ for any parameter far from a small neighborhood of
$a_0$ and {which displays a $C^d$-degenerate unfolding at $a=a_0$
of a homoclinic tangency  of codimension $u$ associated with the
periodic point $Q_{a_0}(f)$.}
\end{prop}
\begin{proof}
By construction, any $g=(g_a)_a\in \mathscr{U}$ is a
$C^d$-degenerate unfolding at $a=a_0$ of a tangency between
$W^s(Q_{a_0},g_{a_0})$ and some local unstable manifold
$W^u_{loc}(x_{a_0},g_{a_0})$ of a point $x_{a_0}\in \Gamma_{a_0}$.
From the assumptions~\ref{H1} and~\ref{H2} we have that both $x_a$
and $Q_a$ belongs to the homoclinic class $H(P_a,g_a)$ of $P_a$
for $g_a$. Moreover, we get that $W^{u}(Q_a,g_a)$ intersects
transversally $W^s_{loc}(x_a,g_a)$ at a point $z_a$ which depends
$C^d$-continuously on $a\in\mathbb{I}^k$. Then
Proposition~\ref{thm-Parametrized-Inclination-Lemma} implies the
existence of discs $D_a$ in $W^u(Q_a,g_a)$ containing $z_a$ such
that the family $D_n=(g^n_a(D_a))_a$ is $C^{d,r}$-close to
$W=(W^u_{loc}(x_a,g_a))$ when $n$ is large.
By a small perturbation, we now will find a new family
$C^{d,r}$-close to $g$, which is a $C^d$-degenerate unfolding at
$a=a_0$ of a homoclinic tangency of codimension $u$ associated
with the continuation of the periodic point $Q_{a_0}$.

We  take local coordinates denoted by $x$ in a neighborhood of
$x_{a_0}$
which correspond to the origin. 
Also denote by $y$ the tangency point between
$W^s(Q_{a_0},g_{a_0})$ and the local unstable manifold
$W^u_{loc}(x_{a_0},g_{a_0})$.
Since the tangency (of dimension $u$) unfolds $C^d$-degenerately,
we have $\vec{p}=(p_a,E_a)_a, \vec{q}=(q_a,F_a)_a\in
C^d(\mathbb{I}^k,G_u(\mathcal{M}))$  such that
\begin{align*}
q_a \in W^s(Q_{a},g_{a}), \quad &\text{and} \quad E_a \subset
T_{q_{a}}W^s(Q_{a},g_{a}), \\
p_a \in W^u_{loc}(x_{a},g_{a}) \quad &\text{and} \quad  F_a
\subset
T_{p_a}W^u_{loc}(x_{a},g_{a}), \\
p_{a_0}=q_{a_0}=y \quad &\text{and} \quad J(\vec{p})=J(\vec{q}).
\end{align*}
Take $\delta>0$ such that, in these local coordinates, the
$2\delta$-neighborhoods of $y$ and its iterations by $g_{a_0}$ and
$g^{-1}_{a_0}$ are pairwise disjoint. Let
$U$ be the $2\delta$-neighborhood of $y$
and assume that $p_a$ and $q_a$ belong to $U$ for all $a$ close
enough to $a_0$. Call $C_a$ the local disc in $W^s(Q_{a},g_{a})$
containing the point $q_a$ and we may suppose that $U$ is such that
the forward iterates of $C_a$, with respect to $g_a$,
are disjoint from each other and from $U$.
Since $D_n=(g^n_a(D_a))_a$ is $C^{d,r}$-close to $W$, we obtain a
$C^{d,r}$-family $\tau_n=(\tau_{n,a})_a$ of diffeomorphisms of
$\mathbb{R}^m$ which sends, in local coordinates,
$g^n(D_a)$ onto $W^u_{loc}(x_a,g_a)$,
is equal to the identity
outside of $U$ and is $C^{d,r}$-close to the constant family
$I=(\mathrm{id}_{\mathbb{R}^m})_a$ as $n\to \infty$.
Let $t_a$ be the point in $D_n\subset W^u(Q_{a},g_{a})$ so that $\tau_{n,a}(t_a)=p_a$.

Consider a $C^\infty$-bump function $ \phi:\mathbb{R}\to
\mathbb{R}$ with support in $[-1,1]$ and equal to 1 on
$[-1/2,1/2]$. Let
$$
\rho: a=(a_1,\dots,a_k) \in \mathbb{I}^k \mapsto
\phi(a_1)\cdot\ldots\cdot \phi(a_k) \in \mathbb{R}.
$$
For a fixed $\alpha>0$, define  the perturbation $g_n=(g_{n,a})_a$
of $g=(g_a)_a$ by
$$
  g_{n,a}= H_{n,a} \circ g_a \quad \text{for all $a\in \mathbb{I}^k$,}
$$
where $H_{n,a}$  in the above local coordinates takes the form
\begin{align*}
  H_{n,a}(x)=x+ \theta \cdot (\tau_{n,a}(x)-x)
  \quad \text{where} \ \
  \theta=\rho\left(\frac{a-a_0}{2\alpha}\right)\phi\left(\frac{\|x-y\|}{2\delta}\right)
\end{align*}
and is the identity otherwise. Observe that if $a\not\in a_0 +
(-2\alpha,2\alpha)^k$ or $x\not\in U$ then $H_{n,a}(x)=x$. In
particular, $g_{n,a}(x)=g_a(x)$ if $a\not\in a_0 +
(-2\alpha,2\alpha)^k$ or $x\not\in g_a^{-1}(U)$.

On the other hand, if $a\in a_0 + (-\alpha,\alpha)^k$ and $x\in
U$, then $H_{n,a}(x)=\tau_{n,a}(x)$. This implies that for $a\in
a_0+(-\alpha,\alpha)^k$, the point $g_a^{-1}(t_a)$ that belongs to
$W^u(Q_a,g_a)$ is sent by $g_{n,a}$ to $p_a=\tau_{n,a}(t_a)$ and
therefore $p_a \in W^u(Q_a,g_{n,a})$. Moreover, since $\tau_{n,a}$
is a $C^r$-diffeomorphism ($r\geq 2$) we also have that $F_a
\subset T_{p_a}W^u(Q_a,g_{n,a})$.

At $a=a_0$ we have that $\tau_{n,a_0}(t_{a_0})=p_{a_0}=q_{a_0}\in
W^s(Q_{a_0}, g_{n,a_0})$ and so the stable and unstable manifolds
of $Q_{a_0}$ for $g_{n,a_0}$ meet at this point. Moreover, since
$\tau_{n,a_0}$ is a $C^r$-diffeomorphism ($r\geq 2$) this
intersection is still non-transverse, i.e., we have a homoclinic
tangency of codimension~$u$. Observe that the perturbed family
$g_{n,a}$ does not affect the disc $C_{a}$ of the stable manifold
of $W^s(Q_a, g_a)$ in $U$. That is, $C_{a}\subset W^s(Q_a,
g_{n,a})$ and $E_a \subset T_{q_a}W^s(Q_a,g_{n,a})$.  Hence, since
from the initial hypothesis $J(\vec{p})=J(\vec{q})$  we get  a
$C^d$-degenerate unfolding at $a=a_0$ of a homoclinic tangency of
codimension $u$ associated with  the hyperbolic periodic point
$Q_{a_0}(g_n)$. Finally, observe
$$
  \|f-g_n\|_{C^{d,r}}=\|(I-H)\circ f\|_{C^{d,r}} \leq \|f\|_{C^{d,r}}
  \|\theta (\tau_n -
  I)\|_{C^{d,r}} \leq \|f\|_{C^{d,r}} o_{\alpha\to 0}(\alpha^{-d}) o_{\delta\to
  0}(\delta^{-r}) \|\tau_n-I\|_{C^{d,r}}.
$$
Since $\|\tau_n-I\|_{C^{d,r}}$ goes to zero as $n \to \infty$, we
can obtain that for a given $\epsilon>0$, there are $n$ large
enough so that $\|f-g_n\|_{C^{d,r}}<\epsilon$.
\end{proof}

\begin{rem} \label{rem-finite-points}
Notice that the perturbation in the previous proposition,
to create the degenerate homoclinic  unfolding at $a_0$, is local
(in the parameter space and in the manifold). Thus, fixing
$N\in\mathbb{N}$ and a finite number of points $a_1,\dots,a_N\in
\mathbb{I}^k$, we can perform the same type of perturbation
inductively and obtain a
dense set of families in $\mathscr{U}$ having degenerate unfoldings at any
$a=a_i$ for
$i=1,\dots,N$.
\end{rem}

The following proposition is an adaptation to the context of
diffeomorphisms of~\cite[Lemma~5.4]{Ber16}. Roughly speaking, this
proposition explains how it is possible to ''stop'' a tangency for
an interval of parameters using a degenerate unfolding,
i.e, how to create a persistent homoclinic
tangency in the language of~\cite{Ber17}.

\begin{prop} \label{prop-paratangecia}
Let $f=(f_a)_a$ be a $k$-parameter $C^d$-family of
$C^r$-diffeomorphisms of a manifold of dimension $m\geq 2$.
Suppose that $f$ is a $C^d$-degenerate unfolding at $a=a_0$ of a
homoclinic tangency of codimension $u>0$  associated with a
hyperbolic periodic point $Q$. Then, for any $\epsilon>0$ there
exists $\alpha_0>0$ such that for every $0<\alpha<\alpha_0$ there
is a $C^{d}$-family $h=(h_a)_a$ of $C^r$-diffeomorphisms such~that
\begin{enumerate}
\item $h$ is $\epsilon$-close to $f$ in the $C^{d,r}$-topology,
\item $h_a=f_a$ for every $a \not \in a_0 + (-2\alpha,2\alpha)^k$,
\item $h_a$ has a homoclinic tangency $Y_a$ of codimension $u$ associated with the
continuation  $Q_a$ of $Q$ for all $a\in a_0+ (-\alpha,\alpha)^k$
and which depend $C^d$-continuously on the parameter $a$.
\end{enumerate}
\end{prop}
\begin{proof}
By assumption $f$ is a $C^d$-degenerate unfolding at $a=a_0$ of a
homoclinic tangency $Y$ associated with a hyperbolic periodic
point $P$. By the definition of degenerate unfolding we have
points $p_a\in W^u(Q_a,f_a)$, $q_a\in W^s(Q_a,f_a)$ and
$c$-dimensional subspaces $E_a$ and $F_a$ of $T_{p_a}W^u(Q_a,f_a)$
and $T_{q_a}W^s(Q_a,f_a)$ respectively such that
\begin{equation} \label{eq:o(ad)}
 d(p_a,q_a)=o(\|a-a_0\|^{d}) \ \ \text{and} \ \
d(E_a,F_a)={o(\|a-a_0\|^{d})} \ \ \text{at $a=a_0$}.
\end{equation}
Here $Q_a$ is the continuation of $Q_{a_0}=P$ for $f_a$. Also
$p_{a_0}=q_{a_0}=Y$ and $(p_a,E_a)$, $(q_a,F_a)$ vary
$C^d$-continuously with respect to the parameter $a\in
\mathbb{I}^k$. We  take local coordinates $x$  in a neighborhood
of $Q$ which correspond to the origin.
By considering an iteration if necessary, we
assume that the tangency point $Y$ belongs to
this neighborhood of local coordinates. Take $\delta>0$ so that
the $2\delta$-neighborhoods of $Y$ and its iterations by $f_{a_0}$
and $f^{-1}_{a_0}$ are pairwise disjoint.
Namely, we will denote by $U$ the $2\delta$-neighborhood of $Y$.
Assume that
$p_a$ and $q_a$ belong to $U$ for all $a$ close enough to $a_0$.
From~\eqref{eq:o(ad)} it follows that
\begin{equation} \label{eq:orden1}
  p_a-q_a=o(\|a-a_0\|^d) \quad \text{at \  $a=a_0$.}
\end{equation}
Observe that the Grasmannian distance between $E_a$ and $F_a$
 is given by the norm of $I-R_a$ restricted
to $F_a$, where $I$ denotes the identity and $R_a$ is the
orthogonal projection onto $E_a$. Then we obtain
from~\eqref{eq:o(ad)} that
\begin{equation} \label{eq:orden2}
I-R_a =o(\|a-a_0\|^d) \quad \text{at \ $a=a_0$}.
\end{equation}
We would like that the rotation occurs around the point $p_a$
and so consider
$\tilde{R}_a x =R_a (x-p_a) + p_a$.
Since $(I-\tilde{R}_a)x=(I-R_a)x+(R_a-I)p_a$ then
from~\eqref{eq:orden2} we still have
\begin{equation} \label{eq:orden3}
 I-\tilde{R}_a =o(\|a-a_0\|^d) \quad \text{at \ $a=a_0$}.
\end{equation}

Consider a $C^\infty$-bump function $ \phi:\mathbb{R}\to
\mathbb{R}$ with support in $[-1,1]$ and equal to 1 on
$[-1/2,1/2]$. Let $$\rho: a=(a_1,\dots,a_k) \in \mathbb{I}^k
\mapsto \phi(a_1)\cdot\ldots\cdot \phi(a_k) \in \mathbb{R}.
$$
For a fixed $\alpha>0$, define the perturbation $h=(h_a)_a$ of
$f=(f_a)_a$ by the relation
$$
  h_{a}= H_{a} \circ f_a \quad \text{for all $a\in \mathbb{I}^k$}.
$$
Here $H_{a}$  in the above local coordinates takes the form
\begin{align*}
  \bar{x}=\big(I-\theta\cdot(I-\tilde{R}_a)\big)\big(x- \theta \cdot
  (q_a-p_a)\big) \quad \text{where} \ \ \theta=\rho\left(\frac{a-a_0}{2\alpha}\right)\phi\left(\frac{\|x-Y\|}{2\delta}\right)
\end{align*}
and is the identity otherwise. Observe that if $a\not\in a_0 +
(-2\alpha,2\alpha)$ or $x\not\in U$ then $H_a(x)=x$. In
particular, $h_a(x)=f_a(x)$ if $a\not\in a_0 + (-2\alpha,2\alpha)$
or $x\not\in f_a^{-1}(U)$. On the other hand, if $a\in a_0 +
(-\alpha,\alpha)$ and $x\in U$ then
$$H_a(x)=\tilde{R}_a\big(x-(q_a-p_a)\big).$$
Since $\tilde{R}_a$ fixes the point $p_a$, we have that
$H_a(q_a)=p_a$. This implies that the point
$q_a^{-1}=f_a^{-1}(q_a)$ that belongs to $W^u(Q_a,f_a)$ is sent by
$h_a$ to $p_a$ which belongs to $W^s_{loc}(Q_a,f_a)$. As the orbit
of $f^{n}_a(p_a)$ for $n\geq 0$ and $f_a^{-n}(q^{-1}_a)$ for $n>0$
never goes through $f_a^{-1}(U)$, then $p_a \in
W^s_{loc}(P_a,h_a)$ and $q_a^{-1}\in W^u(Q_a,h_a)$. Thus, the
stable and unstable manifolds of $Q_a$ for $h_a$ meet at
$p_a=h_a(q_a^{-1})$. Moreover, $F_a^{-1}=Df_a(q^{-1}_a)F_a$ is a
subspace of $T_{q_a^{-1}}W^u(Q_a,f_a)=T_{q_a^{-1}}W^u(Q_a,h_a)$
and
$$
Dh_a(q_a^{-1})F_a^{-1}=
DH_a(q_a)Df_a(q_a^{-1})F_a^{-1}=DH_a(q_a)F_a=D\tilde{R}_a(q_a)
F_a= R_aF_a=E_a.
$$
Since $E_a$ is a subspace of
$T_{p_a}W^s(Q_a,f_a)=T_{p_a}W^s(Q_a,h_a)$, then the intersection
between the stable and unstable manifolds of $Q_a$ for $h_a$ is
tangencial.

To conclude the proposition we only need to prove that for a given
$\epsilon>0$ we can find $\alpha_0$ such that for any
$0<\alpha<\alpha_0$ the above perturbation $h=h(\alpha)$ of $f$ is
actually $\epsilon$-close in the $C^{d,r}$-topology. Notice that
the $C^{d,r}$-norm satisfies
$$
 \|h-f\|_{C^{d,r}}=\|(H-I) \circ f \|_{C^{d,r}}  \leq
  \|I-H\|_{C^{d,r}} \, \|f\|_{C^{d,r}}.
$$
Thus we only need to calculate the $C^{d,r}$-norm of the family
$(I-H_{a})_a$. Since $H_{a}=I$ if $a\not \in a_0 +
(-2\alpha,2\alpha)^k$ or $x\not \in U$ then
\begin{align*}
  \|I-H_{a}\|_{C^{d,r}}\leq\left\| \theta \cdot (I-\tilde{R}_a)(x-\theta\cdot (q_a-p_a))+\theta\cdot (q_a-p_a)\right\|_{C^{d,r}}.
\end{align*}
Since the $C^{d,r}$-norm of $\phi(\|x-Y\|/2\delta)$ is bounded
(depending only on $\delta$), we can disregard this function from
the estimate. Then, to bound the $C^{d,r}$-norms from above it is
enough to show that for $a \in  a_0+(-2\alpha,2\alpha)^k$ the
functions
$$
  F_{\alpha}(a)= \rho(\frac{a-a_0}{2\alpha})(p_a-q_a) \quad
  \text{and} \quad G_{\alpha}(a)  = \rho(\frac{a-a_0}{2\alpha})(I-\tilde{R}_a)
$$
have $C^d$-norm small when $\alpha$ is small enough. But this is
clear from~\eqref{eq:orden1} and \eqref{eq:orden3}, as having into
account that $\|a-a_0\|\leq \alpha$, it follows that
$$
 F_a(a)=\alpha^{-d}\cdot o_{a\to a_0}(\|a-a_0\|^d)=o_{\alpha\to 0}(1) \quad \text{and}
 \quad G_\alpha(a)=\alpha^{-d}\cdot o_{a\to a_0}(\|a-a_0\|^d)=o_{\alpha \to 0}(1).
$$
This completes the proof.
\end{proof}

\begin{rem} \label{rem-paratangency}
Observe that the positive constant $\alpha_0$ in
the above proposition depends initially on the family $f=(f_a)_a$,
the constants $\epsilon>0$, $\delta>0$ and the parameter $a_0\in \mathbb{I}^k$.
The dependence of $a_0$ comes from the function $o_{a\to
a_0}(\|a-a_0\|^d)$ in~\eqref{eq:o(ad)}. However, one can bound
this function by $\nu(\|a-a_0\|)\cdot \|a-a_0\|^d$ where
$\lim_{t\to 0}\nu(t)=0$, controlling the modulus of continuity
$\nu$ of the derivatives of the unfolding. In this form we can get
that $\alpha_0$ does not depend on the parameter $a_0$. Also,
similarly to what was done in the previous proposition, the surgery
using bump functions around a neighborhood of
the initial paratangency point $Y$ can actually be done around any point
$f^N_{a_0}(Y)$
belonging to $W^s_{loc}(P_{a_0},f_{a_0})$.
This allows us to fix a
priori a uniform $\delta>0$ because we only need to control
the distance between one forward/backward iterate of $f^N_{a_0}(Y)$.
Thus,
$\alpha_0$ also does not depend on
$\delta$. Finally, if
$f$ belongs to $\mathscr{U}$
then one can obtain a uniform bound on the continuity modulus
using the fact that we are dealing with compact families of local
stable and unstable manifolds. This proves that in this construction
$\alpha_0$ only depends on $\epsilon$ and $\mathscr{U}$ (i.e, on
the dynamics of the organizing family $\Phi$).
\end{rem}

\begin{rem} \label{rem:simples}
In the case of codimension $u=1$, the tangency $Y_a$ obtained in
the previous proposition could be assumed \emph{simple} in the
sense of~\cite{GST08}.
\end{rem}

Finally, in the next theorem we will show the
existence of Berger domains as stated in
Definition~\ref{def:Berger-domain}. The idea behind the proof is
the replication of the arguments coming from~\cite[Sec.~6.1]{Ber16}
and~\cite[Sec.~7]{Ber17}.
\begin{thm} \label{thm-berger-domain}
There exists a dense subset $\mathscr{D}$ of $\mathscr{U}$ such
that for any $h=(h_a)_a \in \mathscr{D}$ there is a covering of
$\mathbb{I}^k$ by open sets $J_i$
having a persistent  homoclinic tangency
of codimension $u$. That is, $\mathscr{U}$ is a
$C^{d,r}$-Berger domain (of paratangencies of codimension $u$).
\end{thm}
\begin{proof}
Fora a fixed family $g=(g_a)_a\in \mathscr{U}$ and $\epsilon>0$,
Propositions~\ref{prop-dense-of-paratangencies},~\ref{prop-paratangecia}
and Remarks~\ref{rem-finite-points},~\ref{rem-paratangency} imply
the following. We obtain $\alpha_0=\alpha_0(\epsilon)>0$ so that for a fixed
$\alpha$ with $0<\alpha<\alpha_0$ there are points $a_1,\dots,a_N$ in
$\mathbb{I}^k$ such that
\begin{itemize}
 \item[-] the open sets $a_i+(-2\alpha,2\alpha)^k$ for $i=1,\dots,N$
 are pairwise disjoint;
 \item[-] the union of $(a_i+(-2\alpha,2\alpha)^k)\cap \mathbb{I}^k$
 for $i=1,\dots,N$ is dense in $\mathbb{I}^k$;
\item[-] there is a $C^{d,r}$-family
 $h=(h_a)_a$, $\epsilon$-close to $g$, having a persistent homoclinic tangency associated with the continuation of $Q_a$
 for all $a\in (a_i+(\alpha,\alpha)^k)\cap \mathbb{I}^k$ and $i=1,\dots,N$.
\end{itemize}
However this result does not provide an open cover of
$\mathbb{I}^k$. We need to perturbe again $h$ without destroying
the persistent homoclinic tangencies associated with $Q_a$ and
at the same time provide new persistent homoclinic tangencies in the complement
of the union of $(a_i+(\alpha,\alpha)^k)\cap \mathbb{I}^k$
for $i=1,\dots,N$.
To do this, we will need a finite set of different
periodic points $Q^j_a$ to replicate the above argument and ensure
that each new perturbation does not modify the previous one (i.e, there
exists a common $\delta>0$ so that the supports of all the peturbations
are disjoint).
\begin{lem}
There exists a set $\{Q^j_a\}_{j=1,\dots, 3^k}$ of $3^k$
hyperbolic periodic points  satisfying the assumptions~\ref{H1} and~\ref{H2}.
\end{lem}
\begin{proof}
Since the
homoclinic class  $H(Q_a,\Phi_a)$ is not trivial (contains $P_a$),
there exists a horseshoe $\Lambda_a$ containing $Q_a$. Thus,
there are infinitely many different hyperbolic periodic points of $\Phi_a$
whose stable manifold intersect transversely the unstable
manifold of $Q_a$. Then by the inclination lemma, and the
robustness of the folding manifold $S_a$, the stable manifold of
these periodic points also contains a folding manifold.
\end{proof}
Associated with each point $Q^j_a$, as in Proposition~\ref{prop-paratangecia},
there is the paratangency point $Y_j$ where the size of the perturbation is governed by
$\delta_j$. Thus, we can obtain a uniform $\delta$ by taking the infimum over $\delta_j$.
Also corresponding to each $Q^j_a$, there exists a lattice of points
$\{a_i^j\}_{i=1,\dots, N}$ in $\mathbb{I}^k$ such that
the union of the $\alpha$-neighborhoods $a_i^j+(\alpha,\alpha)^k$
in $\mathbb{I}^k$ cover $\mathbb{I}^k$.
That is,
$$\mathbb{I}^k\subset \bigcup_{j=1}^{3^k}\bigcup_{i=1}^N a_i^j+(\alpha,\alpha)^k.$$
Then we can apply Proposition~\ref{prop-paratangecia} independently in $j$ to obtain
the family with the required properties, concluding the proof the the theorem.
\end{proof}
\begin{rem}
The persistent homoclinic tangencies in the open set $J_i$
obtained in the above theorem can be associated with a collection
of saddle periodic points, $Q_a^j$ for $j=1,\dots,3^k$, where all
of them have the same type of multipliers. For instance, we can
take these points being sectionally dissipative or of type~$(1,1)$,
$(2,1)$, $(1,2)$ or $(2,2)$ according to the nomenclature
introduced in~\cite{GST08}.
Also, according
to Remark~\ref{rem:simples}, in the codimension one case, the
persistent homoclinic tangency can be assumed \emph{simple} in the
sense of~\cite{GST08}.
\end{rem}

\section{Proof of Theorem~\ref{thm-thmA}:
 periodic sinks and invariant circles} \label{sec:thm-thmA}
In this section we will prove the $C^{d,r}$-Berger phenomenena of
the coexistence of infinitely many normally hyperbolic attracting
invariant circles and also obtain a similar result for hyperbolic periodic sinks.
For short, we will refer to
both of these types of attractors as \emph{periodic attractors}.

The next proposition claims that every family $f=(f_a)_a$ in
$\mathscr{U}$ can be approximated by a family $g=(g_a)_a$ having a
periodic attractor for every parameter $a$ in $\mathbb{I}^k$.
Moreover, the period of the attractors can be chosen arbitrarily
large. To prove this, since $\mathscr{U}$ is a Berger domain, it
is enough to restrict our attention to the dense set $\mathcal{D}$
of $\mathscr{U}$ having persistent tangencies.

\begin{prop} \label{main-lema}
For any $\epsilon>0$, and every $f=(f_a)_a  \in \mathcal{D}$ there
exists $n_0=n_0(\varepsilon,f)\in\mathbb{N}$ such that for any
$n\geq n_0$ there is a $\epsilon$-close family $g=(g_a)_a$ to
$f=(f_a)_a$ in the $C^{d,r}$-topology satisfying that $g_a$ has a
periodic attractor of period $n$  for all $a\in \mathbb{I}^k$.
Moreover, the attractor obtained for $g_a$ is the continuation of
a (hyperbolic or normally hyperbolic) $n$-periodic attractor
obtained for a map $g_{a_0}$, where $a_0$ belongs to a finite
collection of parameters.
\end{prop}

Before proving the above proposition, we will conclude first from
this result the next main theorem of the section, which in particular
proves Theorem~\ref{thm-thmA}.

\begin{thm} For any $m\in \mathbb{N}$, there exists an open and
dense set $\mathcal{O}_m$ in $\mathscr{U}$ such that for any
family $g=(g_a)_a$ in $\mathcal{O}_m$ there exist positive
integers $n_1<\dots<n_m$ so that $g_a$ has a periodic attractor of
period $n_\ell$ for all $a\in \mathbb{I}^k$ and $\ell=1,\dots,m$.
Moreover, there is a residual subset $\mathcal{R}$ of
$\mathscr{U}$ such that  any $g=(g_a)_a \in \mathcal{R}$ satisfies
that $g_a$ has both infinitely many hyperbolic periodic sinks and
infinitely many normally hyperbolic attracting invariant
circles for all $a\in \mathbb{I}^k$.
\end{thm}
\begin{proof}  First of all consider
the sequence $\epsilon_i=1/i$ for $i\geq 1$. We will prove the
result by induction. To do this, we are going first to construct
$\mathcal{O}_m$ for $m=1$.

By applying Proposition~\ref{main-lema}, for each $f=(f_a)_a$ in
$\mathcal{D}$ taking a sequences of integers $n(i) \geq
n_0(\epsilon_i,f)$ for all $i\geq 1$, we find a $\epsilon_i$-close
family $g=(g_a)_a$ to $f$ such that $g_a$ has a $n(i)$-periodic
attractor for all $a\in \mathbb{I}^k$. Actually, for any parameter
$a$, the attractor that we obtain for $g_a$ is the continuation of
a (hyperbolic or normally hyperbolic) $n(i)$-periodic attractor
obtained for a map $g_{a_0}$ where $a_0$ belongs to a finite
collection of parameters. Thus, from the hyperbolicity of the
attractor, this property persists under perturbations and then we
have a sequence of open sets $\mathcal{O}_1(\epsilon_i,f)$
converging to $f$ where the same conclusion holds for any family
in these open sets. By taking the union of all these open sets for
any $f$ in $\mathcal{D}$ and $\epsilon_i>0$ for $i\geq 1$, we get
an open and dense set $\mathcal{O}_1$ in $\mathscr{U}$ where for
any $g=(g_a)_a\in \mathcal{O}_1$  there exists a positive integer
$n$ such that $g_a$ has an $n$-periodic attractor  for all
parameters $a\in \mathbb{I}^k$.

Now we will assume that $\mathcal{O}_m$ was constructed and show how to
obtain $\mathcal{O}_{m+1}$. Since $\mathcal{O}_m$ is open
and dense set in $\mathscr{U}$ we can start now by taking
$f=(f_a)_a \in \mathcal{O}_m\cap \mathcal{D}$. Hence, there exist
positive integers $n_1<\dots<n_m$ so that $f_a$ has a persistent
$n_\ell$-periodic attractor (a sink or an invariant circle) for
all $a\in \mathbb{I}^k$ for $\ell=1,\dots,m$. As before,
these attractors are the smooth
continuation of periodic attractors centered at a finite collection of parameters.
Therefore, there exists $\epsilon'=\epsilon'(f)>0$
such that any $\epsilon'$-close family $g=(g_a)_a$ still
has the same properties. Then, for any $\epsilon_i<\epsilon'/2$,
we can apply again Proposition~\ref{main-lema}, taking integers
$n(i)_{m+1} > \max\{n_0(\epsilon_i,f),n_{m}\}$ in order to obtain
an $\epsilon_i$-perturbation $g=(g_a)_a$ of $f$ such that $g_a$ has
also a $n(i)_{m+1}$-periodic attractor  for all $a\in
\mathbb{I}^k$. As before, from the persistence of these
attractors, we have a sequence of open sets
$\mathcal{O}_{m+1}(\epsilon_i,f)\subset \mathcal{O}_m$  converging
to $f$ where the same conclusion holds for any family in these
open sets. By taking the union of all these open sets for any
$f\in \mathcal{O}_m\cap \mathcal{D}$ and
$\epsilon_i<\epsilon'(f)$, we get an open and dense set
$\mathcal{O}_{m+1}$ in $\mathscr{U}$ where for any $g=(g_a)_a\in
\mathcal{O}_{m+1}$  there exist positive integers
 $n_1<\dots<n_{m+1}$  such that  $g_a$ has
a $n_\ell$-periodic attractor  for all $a\in \mathbb{I}^k$ for all
$\ell=1,\dots,m+1$.

To conclude the proof of the theorem observe that if $g=(g_a)_a$ belongs to the
residual set $\mathcal{R}=\cap \mathcal{O}_m$ then $g_a$ has
infinitely many of attractors for all $a\in \mathbb{I}^k$.
\end{proof}

Now we will prove Proposition~\ref{main-lema}. To do this we need
the following lemma.

\begin{lem} \label{lema2}
Given $\alpha>0$, let $g=(g_a)_a$ be a $C^{d,r}$-family and assume
that $g_a$ has a simple homoclinic tangency at a point~$Y_a$
(depending $C^d$-continuously on~$a$) associated with a saddle
$Q_a$ satisfying~\eqref{eq1} for any parameter $a\in
a_0+(-\alpha,\alpha)^k$.  Then there exists a sequence of families
$g_n=(g_{na})_a$ approaching $g$ in the $C^{d,r}$-topology such
that $g_{na}=g_a$ if $a\not\in a_0 + (-2\alpha,2\alpha)^k$ and
$g_{na}$ has an $n$-periodic sink or invariant circle
for all $a \in a_0
+ (-\alpha,\alpha)^k$. Moreover, for $n$ large enough
$$
  \|g_n-g\|_{C^{d,r}}=O\left(\frac{\alpha^{-d}}{n}\right).
$$
\end{lem}

Before proving this result, let us show how to get
Proposition~\ref{main-lema} from the above lemma.

\begin{proof}[Proof of Proposition~\ref{main-lema}]
 Given $f=(f_a)_a$ in $\mathcal{D}$,
relabeling and resizing if necessary, we can assume that the cover
of $\mathbb{I}^k$ by the open balls $J_i$ that appears in
Definition~\ref{def:Berger-domain}   is of the form
$$
 \mathbb{I}^k  \subset \bigcup_{j=1}^M \bigcup_{\ell=1}^{N_j}
 J_{j\ell} \quad \text{with} \quad
  J_{j\ell}=a_{j\ell}+(-\alpha_{j\ell},\alpha_{j\ell})^k
  \quad a_{j\ell}\in \mathbb{I}^k \ \ \text{and} \ \ \alpha_{j\ell}>0.
$$
Moreover, the persistent homoclinic tangency $Y_a$ of $f_a$ on
$J_{j\ell}\cap \mathbb{I}^k$ is simple in the sense of~\cite{GST08} and is associated with a saddle $Q_{ja}$
for $j=1,\dots,M$, where for each $j$, the sets
$\overline{J_{j\ell}}$ for $\ell=1,\dots,N_j$ are pairwise
disjoint. To avoid unnecessary complications in the
notation, we can assume that $\alpha_{j\ell}=\alpha$ for all
$j=1,\dots,M$ and $\ell=1,\dots N_j$. Moreover, for each $j$,
the intervals $2J_{j\ell}=a_{j\ell}+(-2\alpha,2\alpha)^k$ are pairwise disjoint with respect to $\ell$.

On the other hand, given $\epsilon>0$, according to
Lemma~\ref{lema2}, we can control the approximation by a function
$F(\alpha,n)$ of order $O(\alpha^{-d}n^{-1})$. We take
$n_0=n_0(\epsilon,f)\in \mathbb{N}$ where
$F(\alpha,n_0)=O(\alpha^{-d}n_0^{-1})<\epsilon$. Now, consider an
integer $n\geq n_0$. We want to find an $\epsilon$-close family
$g=(g_a)_a$ having an $n$-periodic attractor at any parameter $a
\in \mathbb{I}^k$. Having into account that for each $j$ the
intervals $2J_{j\ell}$ are pairwise disjoint,
we can apply Lemma~\ref{lema2} inductively to obtain an
$\epsilon$-close family $g=(g_{a})_a$ to $f$ such that $g_a$ has
an $n$-periodic attractor  for all $a\in J_{j\ell}$. This
concludes the proof.
\end{proof}

Finally to complete the proof we will show Lemma~\ref{lema2}.
However, in order to understand better how periodic sinks
and invariant circles appear in the unfolding of homoclinic
tangencies associated with saddles of the form~\eqref{eq1}, we need
some preliminaries on the theory of rescaling lemmas
from~\cite{GST08}.

\subsection{Rescaling lemma: Generalized Henon map} \label{sec:rescaling-lemma}
Let $f$ be a $C^r$-diffeomorphism of a manifold of dimension
$m\geq 3$ with a homoclinic tangency
  associated with a hyperbolic
periodic point $Q$ with multipliers satisfying the assumptions~\eqref{eq1}. We
assume that the tangency is \emph{simple} in the sense
of~\cite[sec.~1, pg. 928]{GST08}. That is, the tangency is
quadratic, of codimension one and, in the case that the dimension
$m>3$, any extended unstable manifold is transverse to the leaf of
the strong stable foliation which passes through the tangency
point. We need to consider a two-parameter unfolding
$f_\varepsilon$ of $f=f_0$ where
$\varepsilon=(\mu,\varphi-\varphi_0)$ being $\mu$ the parameter
that controls the splitting of the tangency and $\varphi$ the
parameter related to the eigenvalues of $Q$ (here $\varphi_0$ is the value
of $\varphi$ at $\varepsilon=0$). Let $T_0=T_0(\varepsilon)$
denote the local map and in this case this map
corresponds to $f^q_{\varepsilon}$ defined in a neighborhood $W$ of $Q$, where $q$ is the period of $Q$.
By
$T_1=T_1(\varepsilon)$ we denote the map $f_\varepsilon^{n_0}$
defined from a neighborhood $\Pi^-$ of a tangent point $Y^-\in
W^u_{loc}(Q,f_0)\cap W$ of $f_0$ to a neighborhood $\Pi^+$ of
$Y^+=f_0^{n_0}(Y^-)\in W^s_{loc}(Q,f_0)\cap W$. Then, for $n$
large enough, one defines the first return map $T_n=T_1\circ
T_0^n$ on a subset $\sigma_n=T_0^{-n}(\Pi^-)\cap \Pi^+$ of $\Pi^+$
where $\sigma_n\to W^s_{loc}(Q)$ as $n\to\infty$. According
to~\cite[Lemma~1 and~3]{GST08} we have the following result:

\begin{lem}\label{lema-GHM} There exists a sequence of open sets $\Delta_n$ of parameters
converging to $\varepsilon=0$ such that for this values the
first-return $T_n$ has a two-dimensional attracting invariant
$C^r$-manifold $\mathcal{M}_n \subset \sigma_n$
so that after a $C^r$-smooth transformation of
coordinates, the restriction of the map is given by the
Generalized H\'enon map
\begin{equation} \label{eq-GHM}
   \bar{x}=y, \qquad \bar{y}=M-Bx-y^2-R_n(xy+ o(1)).
\end{equation}
 The rescaled parameters $M$, $B$ and $R_n$ are functions of
$\varepsilon \in \Delta_n$ such that $R_n$ converges to zero as
$n\to \infty$ and $M$ and $B$ run over asymptotically large
regions which, as $n\to \infty$, cover all finite values. Namely,
\begin{align*}
M \sim \gamma^{2n}(\mu + O(\gamma^{-n}+\lambda^n)), \quad B \sim
(\lambda \gamma)^{n}\cos(n\varphi+O(1)) \quad \text{and} \quad R_n
= \frac{2J_1}{B}(\lambda^2\gamma)^n
\end{align*}
where $J_1\not=0$ is the Jacobian of the global map $T_1$
calculated at the homoclinic point $Y^-$ for $\varepsilon=0$. The
$o(1)$-terms tend to zero as $n\to \infty$ along with all the
derivatives up to order $r$ with respect to the coordinates
and up to order $r-2$ with respect to the rescaled parameters
$M$ and $B$. Moreover, the limit family is the Henon map.
\end{lem}

The dynamics of the following generalized H\'enon map
\begin{equation} \label{eq-GHM-simplified}
   \bar{x}=y, \qquad \bar{y}=M-Bx-y^2-R_nxy
\end{equation}
was studied in~\cite{GG00,GG04,GKM05} (see also \cite{GGT07}). We
present here the main results for the case of small $R_n$ with
emphasis on the stable dynamics (stable periodic orbits and
invariant circles) in order to apply the corresponding results to
the first return maps coming from \eqref{eq-GHM}. Observe that the difference between equations \eqref{eq-GHM-simplified} (generalized H\'enon map) and \eqref{eq-GHM} (perturbed map) has order $O(R_n)$. Then the existence of stable periodic orbits and
invariant circles for \eqref{eq-GHM} can be inferred from the bifurcation diagram of \eqref{eq-GHM-simplified}. In Figure~\ref{fig-MB} we show the
bifurcation curves for the generalized H\'enon map in \eqref{eq-GHM-simplified}
in the parameter space $(M,B)$.

\begin{figure}\vspace{0.3cm}
\labellist \small\hair 2pt \pinlabel $L^+_n$ at 100 470 \pinlabel
$L^-_n$ at 405 770 \pinlabel $B$ at 490 530 \pinlabel $M$ at 60
790 \pinlabel $\mathrm{HT}_n$ at 180 686 \pinlabel $\mathrm{BT}_n$
at 255 440 \pinlabel $L^\varphi_n$ at 210 570 \pinlabel $L_n$ at
340 570
\endlabellist
\includegraphics[scale=0.6]{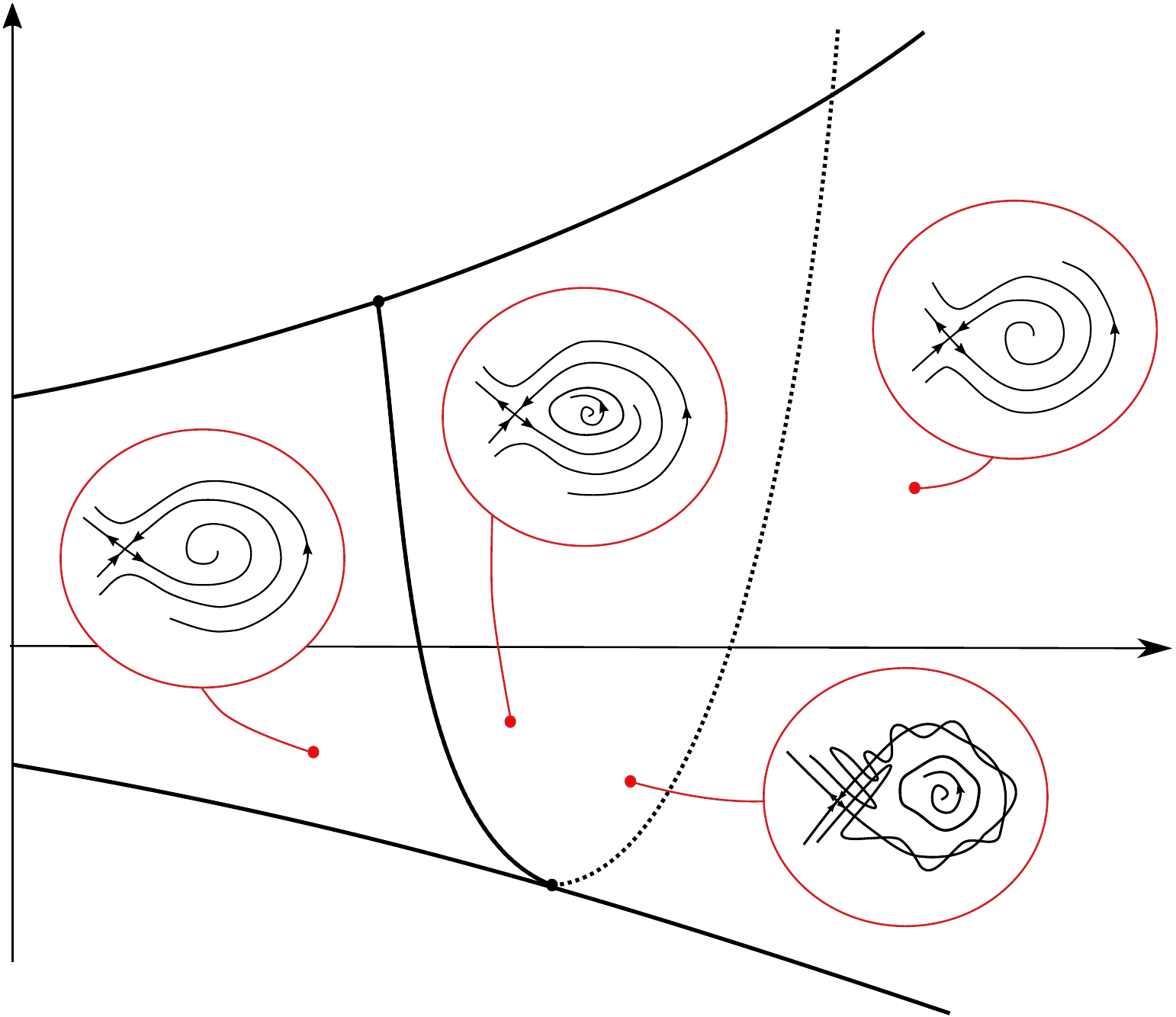}
\caption{Bifurcation curves for the generalized H\'enon
map~\eqref{eq-GHM-simplified} with $R_n > 0$.
   The case $R_n=0$ corresponds with the bifurcation diagram of the H\'enon map.
   In that case, $L_n^\varphi$ collapses with $L_n$ at $B=1$.
   The diagram for the case $R_n<0$ is similar, changing the position of the curves
   $L_n^\varphi$ and $L_n$ and the stability of the invariant circle.      }
\label{fig-MB}
\end{figure}
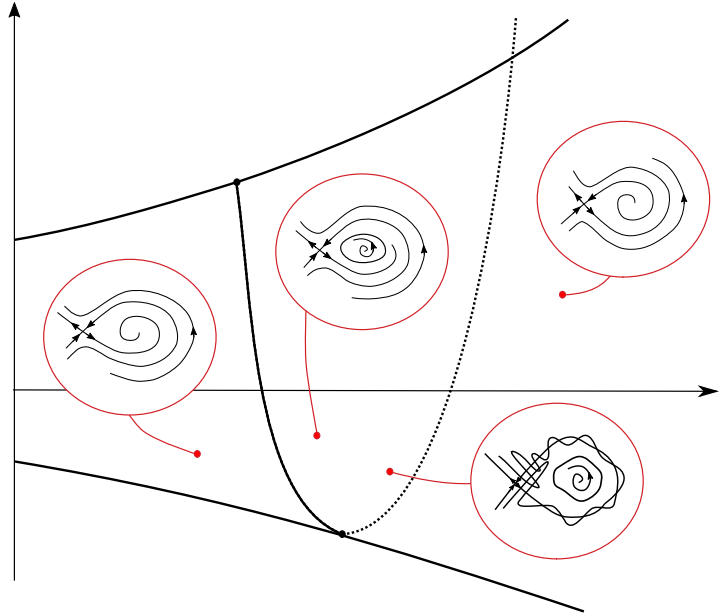

The map~\eqref{eq-GHM-simplified} has in the parameter plane $(M,B)$
the following three bifurcation curves
\begin{align*}
L^+_n \ &: \quad  M = -\frac{(1 +B)^2}{4(1 + R_n)} \\
L^-_n \ &: \quad  M = \frac{1}{4}(1 +B)^2(3 + R_n) \\
L^\varphi_n  \ &: \quad M=\frac{\cos^2 \omega- \cos \omega(2 +
R_n)}{(1+R_n/2)^2}, \quad B = 1+ \frac{R_n\cos \omega}{1 + R_n/2}.
\end{align*}
These correspond to the existence of fixed points having multipliers
on the unit circle: $+1$ at $(M,B) \in L^+_n$; $-1$ at $(M,B)\in
L^-_n$; and $e^{\pm i\omega}$ at $(M,B)\in L^\varphi_n$. Note that
the curve $L^\varphi_n$ is written in a parametric form such that
the  argument $\omega$ ($0 < \omega < \pi$) of the complex
multipliers is the parameter.
We also point out here the points
\begin{equation} \label{eq:pontos}
\begin{aligned}
\mathrm{BT}_n \ &: \quad  M = \frac{-1 -R_n}{(1 + R_n/2)^2},  \quad B=1+\frac{R_n}{1+R_n/2} \\
\mathrm{HT}_n \ &: \quad  M = \frac{3 +R_n}{(1 + R_n/2)^2}, \quad
\ \ \ B=\frac{1-R_n/2}{1+R_n/2}.
\end{aligned}
\end{equation}
The are called as follows: \emph{Bogdanov-Takens point} for
$\mathrm{BT}_n$ and the \emph{Horozov-Takens point} for
$\mathrm{HT}_n$. Also denoted by $L_n$ is an interesting curve
(nonbifurcational) starting at the point $\mathrm{BT}_n$, which
corresponds to the existence of a saddle fixed point
of~\eqref{eq-GHM-simplified} of neutral type (i.e., the fixed
point has positive multipliers whose product is equal to one).
This curve is drawn in Figure~\ref{fig-MB} as the dotted line and
its equation is given by the same expression of
\mbox{$L^\varphi_n$ replacing $\cos \omega$ by $\alpha>1$.}

\begin{figure}
\labellist \small\hair 2pt \pinlabel $K$ at 414 638 \pinlabel $H$
at 310 690 \pinlabel $L_n^-$ at 120 590 \pinlabel $L^\varphi_n$ at
230 445 \pinlabel {\footnotesize$\mathrm{HT}_n$} at 262 575
\endlabellist
  \includegraphics[scale=0.6]{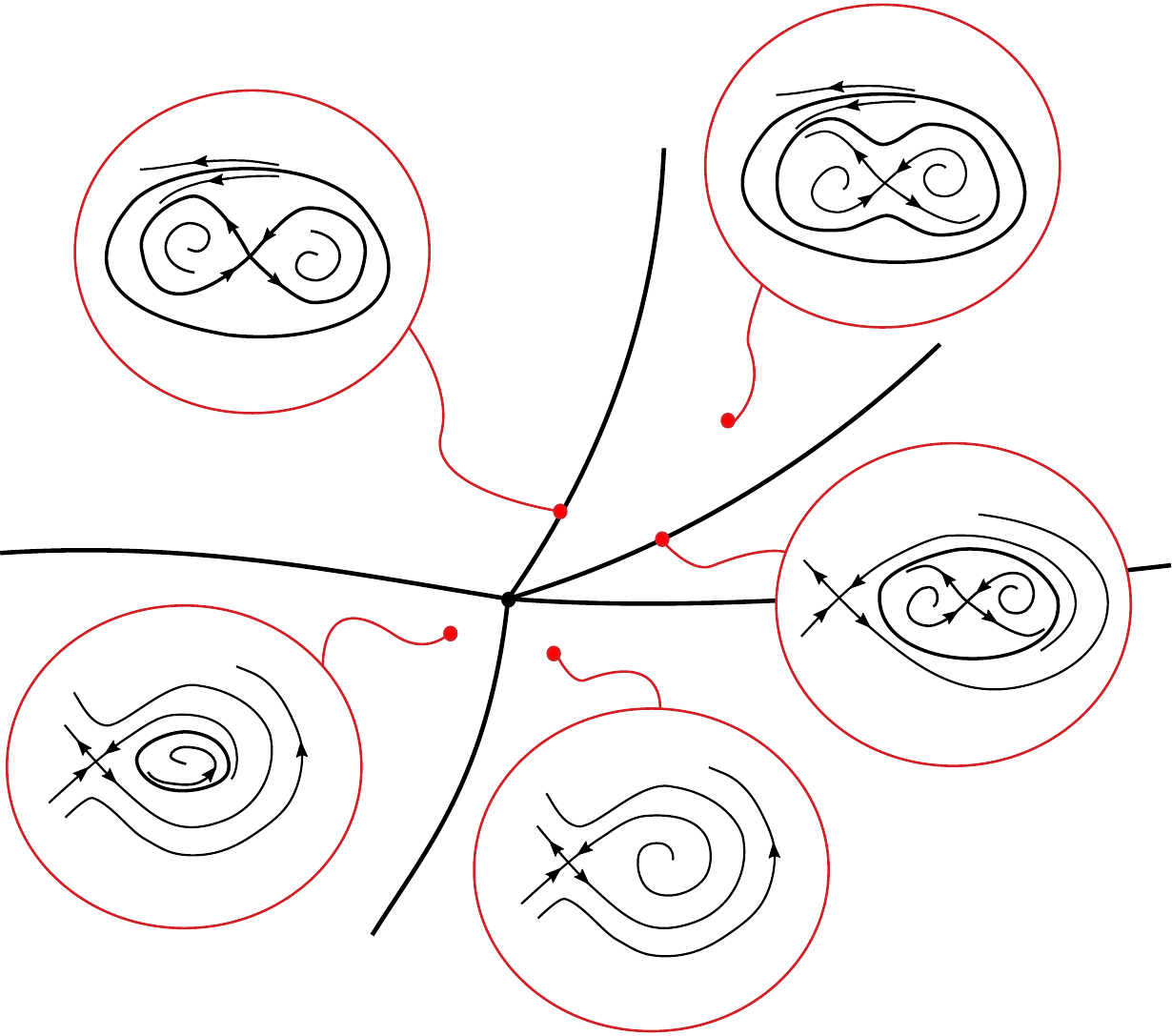}
\caption{Bifurcation diagrams near a Horozov-Takens point for
$R_n<0$. For $R_n=0$ (H\'enon map) the curves $K$ and $H$ collapse
at $B=1$. The diagram for the  case $R_n>0$ is similar but changes
the position of $K$, $H$ and the stability of both invariant
curves. The open domain between $K$ and $H$ is  rather small,
having  the size of a finite order along the $M$-direction and
order $O(R_n)$ along the $B$-direction.} \label{fig-Bminus}
\end{figure}
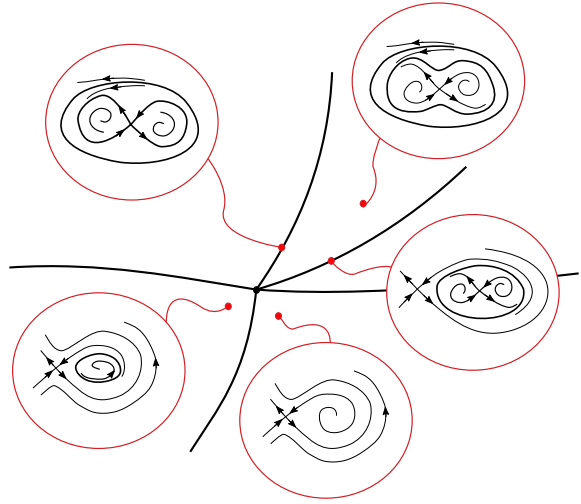

There is an open domain $D^s_n$ , bounded by the curves $L^+_n$,
$L^-_n$ $L^\varphi_n$ with vertices
$\mathrm{BT}_n$, $\mathrm{HT}_n$ (see
Figure~\ref{fig-MB}), such that map~\eqref{eq-GHM-simplified} has
a stable fixed point for parameters in $D^s_n$.
The bifurcations of periodic points
with multipliers $e^{\pm i\omega}$
can lead
to asymptotically stable or/and unstable invariant circles.
The first return map $T_n$ has an
invariant circle which is either stable if $R_n>0$ or unstable if
$R_n<0$. Observe that the sign of $R_n$ actually only depends on
$J_1\not=0$. Thus, if $J_1>0$ we obtain parameter values where we
have a stable close invariant curve. In the case $J_1<0$, the
existence of a stable closed invariant curve
in~\eqref{eq-GHM-simplified} follows from the bifurcation analysis
of the Horozov-Takens point $\mathrm{HT}_n$. Some
of the elements that appear in this non-degenerate bifurcation are
showen in Figure~\ref{fig-Bminus}. Actually, when $J_1\not=0$ (i.e,
$R_n\not =0$), near the point $\mathrm{HT}_n$
there are open domains parameter values where stable and
unstable closed invariant curves coexist.

Moreover, for any map that
is $O(R_n)$-close to to~\eqref{eq-GHM-simplified} in the
$C^3$-topology the corresponding bifurcations still remain
non-degenerate and preserve the same stability. Thus, we can
obtain the same type of results for the pertubed map~\eqref{eq-GHM}.
To summarize for future reference,  for $n$ large
enough we find open sets $A^1_n$, $A^2_n$ of $(M,B)$-parameters accumulating at
$\mathrm{BT}=(-1,1)$ and $\mathrm{HT}=(3,1)$ respectively as $n\to
\infty$ such that the following holds. If
$(M,B)\in A^1_n$ (resp.~$(M,B)\in A^2_n$) then $T_n$ has a
hyperbolic attracting periodic point (resp.~a normally hyperbolic
attracting invariant circle).

%
%

\subsection{Proof of Lemma~\ref{lema2}}
By assumption the map $g_a$  has a homoclinic tangency at a point
$Y_a$ associated with a sectional dissipative periodic point $Q_a$
for all $a\in a_0+(-\alpha,\alpha)^k$. Actually, the tangency must
be smoothly continued until $\|a-a_0\|_{\infty}=\alpha$. We
consider a two-parameter unfolding $g_{a,\varepsilon}$ of the
homoclinic tangency $Y_a$ of $g_a$ for $a\in a_0+
[-\alpha,\alpha]^k$,
where $\varepsilon=(\mu,\varphi)$.
Here $\mu$ is the parameter that controls the splitting of the
tangency and $\varphi$ is the perturbation of the argument of the complex
eigenvalues of $Q_a$.
We can take local coordinates $(x,u,y)$  with $x\in \mathbb{C}$,
$u\in\mathbb{R}^{m-3}$ and $y\in \mathbb{R}$ in a neighborhood of
$Q_a$ which corresponds to the origin, such that $W^s_{loc}(Q_a)$
and  $W^u_{loc}(Q_a)$ acquire the form $\{y=0\}$ and $\{x=0,u=0\}$
respectively. Moreover, the complex eigenvalues related to the stable
manifold of $Q_a$ correspond to the $x$-variable and
the tangency point $Y_a$
is represented by $(x^+,u^+,0)$.

Let us consider a $C^\infty$-bump
function $ \phi:\mathbb{R}\to \mathbb{R}$ with support in $[-1,1]$
and equal to 1 on $[-1/2,1/2]$. Let $$\rho: a=(a_1,\dots,a_k) \in
\mathbb{I}^k \mapsto \phi(a_1)\cdot\ldots\cdot \phi(a_k) \in
\mathbb{R}.$$ Take $\delta>0$ so that the $\delta$-neighborhoods
in local coordinates of $Q_a$ and $g_a^{-1}(Y_a)$ are disjoint.
Observe that
these two neighborhoods, call $U$ and $V$ respectively, can be taken
independent of $a$.  We can write
$$g_{a,\varepsilon}= H_{a,\varepsilon}
\circ g_a,$$ where $H_{a,{\varepsilon}}$  in these local coordinates
takes the form
\begin{align*}
   \bar{x}&=\left(1-\rho\left(\frac{a-a_0}{2\alpha}\right)\phi\left(\frac{\|(x,u,y)\|}{2\delta}\right)
   (1-e^{i\varphi})\right) x
   \\
   \bar{u}&=u \\
   \bar{y}&=y+\rho\left(\frac{a-a_0}{2\alpha}\right)\phi\left(\frac{\|(x,u,y)-(x^+,u^+,0)\|}{2\delta}\right)\mu,
\end{align*}
and is the identity otherwise.
Observe that if $a\not\in a_0 + (-2\alpha,2\alpha)^k$ then
$g_{a,\varepsilon}=g_a$ and if $(x,u,y)\not \in U\cup V$ then
$g_{a,\varepsilon}=g_a$.

Recall in Section~\ref{sec:rescaling-lemma} the definition of
the first return map associated with the unfolding of a simple
homoclinic tangency. Since the tangency point $Y_a$ depends
$C^d$-continuously on $a_0+[-\alpha,\alpha]^k$, we may assume that the
first-return map $T_n=T_n(a,\varepsilon)$ also depends smoothly as
a function of the parameter $a$ on $a_0+[-\alpha,\alpha]^k$.

\begin{lem}[Parametrized rescaling lemma]\label{para-lema-GHM}
There exist
families of open sets $(\Delta_n(a))_a$ of parameters converging to
$\varepsilon=0$ as $n\to\infty$ such that for any $\varepsilon \in
\Delta_n(a)$ the map $T_n=T_n(a,\varepsilon)$ has an attracting
$C^r$-manifold $\mathcal{M}_n=\mathcal{M}_n(a,\varepsilon)$ for any $a\in a_0+[-\alpha,\alpha]^k$.
Moreover, there exists a $C^{d,r}$-family of transformations of coordinates which bring the first-return map $T_n$ restricted to $\mathcal{M}_n$ to the form
given by~\eqref{eq-GHM}
\begin{equation} \label{eq-GHM2}
   \bar{x}=y, \qquad \bar{y}=M-Bx-y^2-R_n(xy+ o(1)).
\end{equation}
Here the rescaled parameters $M=M_n(a,\varepsilon)$ and $B=B_n(a,\varepsilon)$ are at least $C^d$-smooth
functions (recall that $d\leq r-2$) on
$$
\Delta_n=\left\{(a,\varepsilon):  a\in a_0+(-\alpha,\alpha)^k \
\text{and} \ \varepsilon \in \Delta_n(a)\right\}.
$$
The same
property holds for the coefficient $R_n=R_n(a,\varepsilon)$ and the $o(1)$-terms. More
specifically,
\begin{align}  \label{eq:MB}
M \sim \gamma^{2n}_a(\mu + O(\gamma_a^{-n}+\lambda_a^n)), \quad
B \sim (\lambda_a
\gamma_a)^{n}\cos(n\varphi+O(1))
 \quad \text{and}
\quad R_n = \frac{2J_{1a}}{B}(\lambda^2_a\gamma_a)^n
\end{align}
where $\lambda_a=\lambda_a(g)$ and $\gamma_a=\gamma_a(g)$ are the
eigenvalues of $Q_a=Q_a(g)$  satisfying~\eqref{eq1} for $g_a$.
\end{lem}

\begin{proof}
Let us analyze the proof of the rescaling lemma in~\cite{GST08}, more specifically
the change of coordinates for the first return maps given in
Section 3.2~\cite{GST08} for the case $(2,1)$ with $\lambda\gamma>1$.
From equations ~\cite[Eq.~(3.12)-(3.16)]{GST08} we can observe the all the transformations of coordinates can be performed smoothly on the parameter
$a \in a_0+[-\alpha,\alpha]^k$. The exponents that will appear in the orders of convergence will not depend on the parameter $a$ but only on $n$. On the other hand the constants in the $O$-terms will depend on the parameter but these can be uniformely bounded due to the compactness of the parameter space.
These considerations allow us basically to apply
the rescaling Lemma~\ref{lema-GHM} smoothly on the parameter $a \in
a_0+[-\alpha,\alpha]^k$.
\end{proof}

\begin{lem} \label{Delta_expand}
For
$n$ large enough, $\Delta_{n}(a)$ can be taken in such a way that
$$
 \phi_{n,a}: \Delta_{n}(a) \to (-10,10)^2 \setminus B(0,r), \qquad \phi_{n,a}(\epsilon)
=(M_n(a,\epsilon), B_n(a,\epsilon))
$$
is a diffeomorphism where $M_n$ and $B_n$ are the functions given
in~\eqref{eq:MB} for fixed $n$ and $a$. Here $B(0,r)$ is a closed
ball of some small fixed radius $r$ around the origin.
\end{lem}

\begin{proof}
Let us introduce the parameter value $\varepsilon^0_n(a)=(\mu_n^0(a),\varphi^0_n(a))$,
which by definition satisfies $\phi_{n,a}(\varepsilon^0_n(a))=0$. It can be seen from
the expressions of $M$ and $B$ in~\eqref{eq:MB} that $\mu^0_n(a)=O(\gamma^{-n}_a+\lambda^n_a)$ and $\varphi^0_n(a)= \frac{\pi}{2n} +O(1/n)$. We take out $B(0,r)$ around the origin in $(-10,10)^2$ so that $B$ is bounded away from zero and $R_n$ is well-defined. Similar values were considered in~\cite[pg.~946]{GST08} and as is
claimed there, the rescaled functions $M$ and
$B$ can take arbitrarily finite values when $\mu$ varies close to
$\mu^0_n(a)$ and $\varphi$ near
$\varphi^0_n(a)$. Let us explain this.
Although the $O$-functions in~\eqref{eq:MB}  depend on
$\varepsilon$, the functions $M$ and $B$ basically only depend on
$\mu$ and $\varphi$ respectively for $n$ large enough. Actually
on~\cite[pg.~946]{GST08} are giving explicit expressions for
$M_{n,a}$ and $B_{n,a}$.
Using them one can
observe that
$$
 \partial_\mu M_{n,a} \sim \gamma_a^{2n} \not = 0 \quad \text{and} \quad
 \partial_\varphi B_{n,a} \sim n(\lambda_a\gamma_a)^n
 \sin(n\varphi+O(1)) \not = 0
$$
for all $\varepsilon=(\mu,\varphi)$ close to $\varepsilon^0_n(a)$.
Then the Jacobian of $\phi_{n,a}$ converges to infinity, uniformly
on $a\in a_0+[-\alpha,\alpha]^k$  as $n\to \infty$. The rate of
growth is exponential. This implies that $\phi_{n,a}$ is an
invertible expanding map with arbitrarily large uniform expansion
on $a\in a_0 + [-\alpha,\alpha]^k$. On the other hand, the size of
$\Delta_{n}(a)$ (coming mainly from considerations on the angle
$\varphi$) where the expanding map $\phi_{n,a}$ is defined has
decay of order $O(1/n)$. Thus, for $n$ large enough we get that a
neighborhood of $\varepsilon^0_n(a)$ can be taken so that its
image under $\phi_{n,a}$ is $(-10,10)^2$. In particular we can
take $\Delta_n(a)$ being diffeomorphic to $(-10,10)^k\setminus
B(0,r)$.
\end{proof}
Consequently,
$$
\Phi_n(a,\varepsilon)=(a,\phi_{n,a}(\varepsilon))
$$
is a diffeomorphism between the set $\Delta_n$ defined above and
$(a_0+[-\alpha,\alpha]^k)\times (-10,10)^2\setminus B(0,r)$.
On the other hand, although the coefficient $R_n$ depends on $B$,
note that the range of values it takes is negligible when $B$ is
bounded from zero and $n$ is large enough. Actually from the
relations in ~\eqref{eq:MB} it follows that $R_n=o(1)$. Thus, the
bifurcation diagram of~\eqref{eq-GHM2} can be studied from the
results described in Section~\ref{sec:rescaling-lemma} assuming
$R_n=o(1)$ independent of $B$.

Let us remind the reader of the Bogdanov-Takens $\mathrm{BT}_n(a)$
and the Horozov-Takens $\mathrm{HT}_n(a)$ points given
in~\eqref{eq:pontos}, which now also depend on $a$ and accumulate
at $\mathrm{BT}=(-1,1)$ and $\mathrm{HT}=(3,1)$ respectively as
$n\to \infty$. Hence, according to
Section~\ref{sec:rescaling-lemma}  for each $n$ large enough we
find open subsets $A^1_n(a)$, $A^2_n(a)$ in the $(M,B)$-parameter
plane such that if $(M,B)\in A^1_n(a)$ (resp.~$(M,B)\in
A^2_n(a)$), then $T_n$ has a hyperbolic attracting periodic point
(resp.~
attracting
smooth invariant circle) for all $a \in a_0+[-\alpha,\alpha]^k$.
Moreover, we can assume the points $\mathrm{BT}_n(a)$ and
$\mathrm{HT}_n(a)$
belong to the boundary of $A^1_n(a)$ and
$A^2_n(a)$ respectively. Since these sets vary $C^d$-continuously
with respect to the parameter $a \in [-\alpha,\alpha]^k$, we can
choose $C^d$-continuously $(M^*_n(a),B^*_n(a)) \in A^*_n(a)$ for
$*=1,2$ arbitrarily close to $\mathrm{BT}_n(a)$ and
$\mathrm{HT}_n(a)$ respectively.

Since $\Phi_n$ is a diffeomorphism, we may find
a $C^{d}$-function
$\varepsilon^*_n(a)=(\mu_n^*(a),\varphi^*_n(a))$ for $a \in
a_0+[-\alpha,\alpha]^k$, $*=1,2$, defined as
$\Phi_n^{-1}(a,(M^*_n(a),B^*_n(a))=(a,\varepsilon^*_n(a))$. In
particular,
\begin{equation} \label{eq2}
\begin{aligned}
M^*_n(a)&=M_{n,a}(\varepsilon^*_n(a))\sim \gamma^{2n}_a(\mu_n^*(a) +
O(\gamma_a^{-n}+\lambda_a^n)) \\ 
B^*_n(a)&=B_{n,a}(\varepsilon^*_n(a)) \sim (\lambda_a
\gamma_a)^{n}\cos(n\varphi^*_n(a)+O(1)).
\end{aligned}
\end{equation}
Extending smoothly $\varepsilon^*_n(a)$ to $\mathbb{I}^k$ we can
consider the sequence of families
$\tilde{g}_n=(\tilde{g}_{n,a})_a$ where
$$
   \tilde{g}_{n,a}=g_{a,\varepsilon^*_n(a)} \qquad \text{for  $a\in
   \mathbb{I}^k$ and $n$ large enough}.
$$
Observe that $\tilde{g}_{n,a}=g_a$ for $a \not \in a_0 +
(-2\alpha,2\alpha)^k$ and we may assume
has an $n$-periodic attractor (a sink or an invariant circle) for
all $a\in a_0+(-\alpha,\alpha)^k$.

To conclude the proof of the lemma  we only need to show that
$\tilde{g}_n$ converges to $g$ in the $C^{d,r}$-topology. To do
this, notice that the $C^{d,r}$-norm satisfies
$$
 \|\bar{g}_n-g\|=\|(I-H_{a,\varepsilon^*_n(a)})\circ g_a \|  \leq
  \|I-H_{a,\varepsilon^*_n(a)}\| \, \|g\|
$$
where $I$ denotes the identity and  the $C^{d,r}$-norm of any $g=(g_a)_a$ in
the Berger domain $\mathscr{U}$ is bounded.
Thus, we
only need to calculate the $C^{d,r}$-norm of the family
$(I-H_{a,\varepsilon^*_n(a)})_a$. Since
$H_{a,\varepsilon^*_n(a)}=I$ if $a\not \in a_0 +
(-2\alpha,2\alpha)^k$ or $(x,u,y)\not \in U \cup V$, therefore
$\|I-H_{a,\varepsilon^*_n(a)}\|$ is less or equal than
\begin{align*}
  \left\|\rho\left(\frac{a-a_0}{2\alpha}\right)\phi\left(\frac{\|(x,u,y)\|}{2\delta}\right)
  (1-e^{i\varphi^*_n(a)})x\right\| +
  \left\|\rho\left(\frac{a-a_0}{2\alpha}\right)\phi\left(\frac{\|(x,u,y)-(x^+,u^+,0)\|}{2\delta}\right)\mu^*_n(a)\right\|.
\end{align*}
To estimate the $C^{d,r}$-norms above it suffices to show
that the functions
$$
  F_{n}(a)= \rho(\frac{a-a_0}{2\alpha})(1-e^{i\varphi^*_n(a)}) \quad
  \text{and} \quad G_{n}(a)  = \rho(\frac{a-a_0}{2\alpha})\mu^*_n(a)
$$
have $C^d$-norm small when $n$ is large and $a \in  a_0+(-2\alpha,2\alpha)^k$. Actually we will prove
the following:
\begin{equation} \label{eq:order}
\|F_{n}\|_{C^d}= O\left(\frac{\alpha^{-d}}{n}\right) \quad
\text{and} \quad \|G_{n}\|_{C^d}=
O\left(\frac{\alpha^{-d}}{n}\right)
\end{equation}
Observe that this assertion completes the proof.
To prove this we will need the following derivative estimates on the functions $\mu_n^*(a)$ and $\varphi_n^*(a)$. Here the symbol $\partial_a^j$ is used to denote the $j$-th partial
derivative with respect to  the coordinates $a_i$ of $a$ using the multi-index notation.

\begin{lem} \label{Est_der} For $1\leq |j|\leq d$
\begin{enumerate}[label=(\roman*)]
\item $\mu_n^*(a)=O(\gamma_a^{-n}+\lambda_a^n) \ \ \text{and}  \ \ \partial_a^j\mu_n^*(a)=O(n^{|j|}(
\gamma_a^{-n}+\lambda_a^n)).$ \\[-0.3cm]
\item $\varphi_n^*(a)=O\left(\frac{1}{n}\right), \ \
\partial_a^j\varphi_n^*(a)=O\left(\frac{n^{{|j|-1}}}{(\gamma_a\lambda_a)^{n}}\right) \ \ \text{and}
\ \
\partial_a^j(e^{i\varphi_n^*(a)})=O\left(\frac{n^{{|j|-1}}}{(\gamma_a\lambda_a)^{n}}\right).$
\end{enumerate}

\end{lem}

Assuming the above lemma let us now prove the estimates in~\eqref{eq:order}, starting with the second one. To do
this, using the Leibniz formula,
$$
\partial^\ell_a G_{n}(a)=\sum_{j:j\leq\ell} \binom{\ell}{j} \ \partial_a^{\ell-j}
\rho\left(\frac{a-a_0}{2\alpha}\right) \cdot \partial_a^j\mu_n^*(a).
$$
Substituting the estimate from Lemma \ref{Est_der} (i) in the above expression we
obtain that
$$\partial_a^\ell G_{n,\alpha}(a)=O(\alpha^{-d}\cdot n^{\left|\ell \right|}
(\gamma_a^{-n}+\lambda_a^n))$$
which, in fact, implies a better
estimate than~\eqref{eq:order}.

To prove  the first estimate in~\eqref{eq:order} for $\|F_{n}\|_{C^d}$, using again the Leibniz formula
$$
\partial_a^\ell F_{n,\alpha}(a) =
(1-e^{i\varphi_n^*(a)}) \ \partial_a^\ell\rho\left(\frac{a-a_0}{2\alpha}\right)
-\sum_{j:\ 0<j\leq\ell} \binom{\ell}{j}\, \
 \partial_a^{j}(e^{i\varphi_n^*(a)}) \ \partial_a^{\ell-j}\rho\left(\frac{a-a_0}{2\alpha}\right).
 $$
Applying Lemma~\ref{Est_der} (ii),
$$
\partial_a^\ell F_{n,\alpha}(a) =O\left(\frac{\alpha^{-\left|\ell\right|}}{n}\right)
+\sum_{j:\ 0<j\leq\ell} \binom{\ell}{j}\, O\left(
\frac{n^{{|j|-1}}}{(\gamma_a\lambda_a)^{n}} \cdot
\alpha^{-\left|\ell-j\right|} \right).
$$
Thus, $\partial_a^\ell F_{n,\alpha}(a)=
O\left(\alpha^{-d}n^{-1}\right) $. To complete the proof of Lemma
\ref{lema2} we have to show the estimates of Lemma \ref{Est_der}.
\begin{proof} [Proof of Lemma \ref{Est_der}]
First we will prove the estimates on $\mu_n^*(a)$ and
$\partial_a^j\mu_n^*(a)$. Let us first observe that $M$, up to a
multiplicative factor, is actually equal to  $\gamma^{2n}_a(\mu +
O(\gamma_a^{-n}+\lambda_a^n))$ (see the exact expressions for $M$
in ~\cite[Section 3.2]{GST08}). Although this multiplicative
factor depends on the perturbation $\varepsilon_n$, to avoid
notational clutter we will assume in what follows and without loss
of generality that this factor is always $1$. Then $M^*_n(a)$ can
be written as
\begin{equation} \label{Eq_M}
M^*_n(a)=\gamma^{2n}_a\mu_n^*(a) + h_n(a)
\end{equation}
where $h_n(a)$ is independent of $\mu_n^*(a)$ and
$h_n(a)=O(\gamma_a^{-n}+\lambda_a^n)$. Again using the explicit
expressions for $M$ in ~\cite[Section 3.2]{GST08}, one can see
that $\partial_a^j h_n(a)=O(n^{|j|}(\gamma_a^{-n}+\lambda_a^n))$.

\begin{claim} \label{claim_M}
\begin{equation}  \label{covergence_M}
M^*_n=O(1) \quad \text{and} \quad \partial_a^j M^*_n=o(1) \ \
\text{for all $1\leq |j|\leq d$}
\end{equation}
\end{claim}
\begin{proof}[Proof of Claim~\ref{claim1}]
The expressions of $\mathrm{BT}_n(a)$ and $\mathrm{HT}_n(a)$  in
\eqref{eq:pontos} imply that these functions have order $O(R_n)$
and analogously their respective derivatives have order
$O(\partial_a^j(R_n))$. On the other hand $M^*_n(a)$
 is
$O(R_n)$-close to $\mathrm{BT}_n(a)$ or $\mathrm{HT}_n(a)$, see Section~\ref{sec:rescaling-lemma}. Also the derivatives of $\partial_a^j M^*_n$ are $O(\partial_a^j(R_n))$-close. Therefore,
$$M^*_n=O(1)+O(R_n) \quad \text{and} \quad \partial_a^j M^*_n=O(\partial_a^j R_n).$$
Finally, from the equation for $R_n$ given in Lemma~\ref{para-lema-GHM}, we have that
$R_n=o(1)$ and $\partial_a^j R_n=o(1)$. Combining this with the previous estimates proves the claim.
\end{proof}
In particular, from~\eqref{Eq_M} we obtain that
$$\gamma^{2n}_a \mu_n^*(a)=O(1) + O(\gamma_a^{-n}+\lambda_a^n),$$
which implies
$\mu_n^*(a)=O(\gamma^{-2n}_a)=O(\gamma_a^{-n}+\lambda_a^n)$. To
show the estimates on $\partial_a^j\mu_n^*(a)$ we will proceed by
an inductive argument on $|j|$. In the case $|j|=1$, taking the
derivate of \eqref{Eq_M} gives
$$
\partial_a^j M^*_n(a)=\partial_a^j(\gamma^{2n}_a)\mu_n^*(a) +\gamma^{2n}_a \partial_a^j\mu_n^*(a)+
\partial_a^j h_n(a).
$$
We have that $\mu_n^*(a)=O(\gamma_a^{-n}+\lambda_a^n)$,
$\partial_a^j M^*_n(a)=o(1)$ from Claim~\ref{claim_M} and
$\partial_a^j h_n(a)=O(n(\gamma_a^{n}+\lambda_a^n))$. Also
$\partial_a^j(\gamma^{2n}_a)=2n\gamma^{2n-1}_a\cdot
\partial_a^j(\gamma_a)=O(n\gamma^{2n}_a)$. Combining all these
estimates with the previous equation implies
$$\gamma^{2n}_a \partial_a^j\mu_n^*(a)=O(n\gamma^{2n}_a)\cdot O(\gamma_a^{-n}+\lambda_a^n)
+O(n(\gamma_a^{-n}+\lambda_a^n))+o(1).$$ Then
$\partial_a^j\mu_n^*(a)=O(n( \gamma_a^{-n}+\lambda_a^n))$, which
proves the formula for $|j|=1$. To prove the necessary expression
for any multi-index $\ell$, we will proceed by induction assuming
that $\partial_a^j\mu_n^*(a)=O(n^{|j|}(
\gamma_a^{-n}+\lambda_a^n))$ for any $j$ with $1\leq |j|
<\left|\ell\right|$ and will show the same estimate for $\ell$.
From~\eqref{Eq_M} and using the Leibniz formula we obtain
\begin{equation} \label{eq_M2}
 \partial_a^{\ell}M^*_n(a) =
\sum_{j:j\leq\ell} \binom{\ell}{j} \ \partial_a^{\ell-j}
(\gamma^{2n}_a) \cdot \partial_a^j\mu_n^*(a)+\partial_a^{\ell}h_n(a).
\end{equation}
Now $\partial_a^{\ell-j}
(\gamma^{2n}_a)=O(n^{\left|\ell-j\right|}\gamma^{2n}_a)$ and so
the order of $\partial_a^{\ell-j}(\gamma^{2n}_a) \cdot \partial_a^j\mu_n^*(a)$ is
$$O(n^{\left|\ell-j\right|}\gamma^{2n}_a)\cdot O(n^{|j|}( \gamma_a^{-n}+\lambda_a^n))
=O(n^{\left|\ell\right|}\gamma^{2n}_a(
\gamma_a^{-n}+\lambda_a^n)).$$ Also
$\partial_a^{\ell}M^*_n(a)=o(1)$ by Claim~\ref{claim_M} and
$\partial_a^{\ell}h_n(a)=O(n^{\left|\ell\right|}(\gamma_a^{n}+\lambda_a^n))$.
Then isolating the term with
$j=\ell$ from the sum in~\eqref{eq_M2} we get that
$$\gamma^{2n}_a\cdot\partial_a^\ell\mu_n^*(a)=
O(n^{\left|\ell\right|}\gamma^{2n}_a( \gamma_a^{-n}+\lambda_a^n))
+O(n^{\left|\ell\right|}(\gamma_a^{-n}+\lambda_a^n))+o(1).$$ Then
we may conclude that
$\partial_a^\ell\mu_n^*(a)=O(n^{\left|\ell\right|}(
\gamma_a^{-n}+\lambda_a^n))$ proving item~(i) of
Lemma~\ref{Est_der}.

Now we will prove the second part of the lemma, that is, the
estimates on $\varphi_n^*(a)$ and $\partial_a^j\varphi_n^*(a)$. As
was mentioned before in Lemma~\ref{Delta_expand} because of the
expressions in~\cite[pg.~946]{GST08}, we have that
$\varphi^0_n(a)= \frac{\pi}{2n} +O(1/n)$ and the size of
$\Delta_{n}(a)$ has decay of order $O(1/n)$. Thus, we may also
conclude that $\varphi_n^*(a)=O(n^{-1})$. Now assuming
$\partial_a^j\varphi_n^*(a)=O\left(n^{{|j|-1}}(\gamma_a\lambda_a)^{-n}\right)$,
it is not hard to see by an inductive argument on the derivatives
that also
$\partial_a^j\left(e^{i\varphi_n^*(a)}\right)=O\left(n^{{|j|-1}}(\gamma_a\lambda_a)^{-n}\right)$.

In what follows we will prove that $\partial_a^j\varphi_n^*(a)=
O(n^{|j|-1}(\gamma_a\lambda_a)^{-n})$ for $j\geq 1$. This will be
done by induction in $|j|$, starting with $|j|=1$. To do this,
according to~\eqref{eq2} and writting
$h_n(a)=n\varphi_n^*(a)+O(1)$, we have that $B^*_n(a)
\sim(\lambda_a\gamma_a)^{n}\cos(h_n(a))$. Having in mind the exact
expression for the function $B$ given in~\cite[pg.~946]{GST08},
$B_n(a)$ differs from $(\lambda_a\gamma_a)^{n}\cos(h_n(a))$ by a
multiplicative constant $b_0$, that depends on $\varepsilon_n$.
Similarly to what was done for $M^*_n(a)$ and to avoid unnecessary
notational complications we will assume, without loss of
generality, that this factor is always $1$. Then,
\begin{equation} \label{B*}
B^*_n(a)=(\lambda_a\gamma_a)^{n}\cos(h_n(a))
\end{equation} and
derivating both sides of the equation with respect to some index
$j$, $|j|=1$, we obtain
\begin{align}
\label{eq:O1}  \partial_a^j B^*_n(a) &= n \cdot B^*_n(a) \cdot \partial_a^j (\log
\lambda_a \gamma_a) + (\lambda_a \gamma_a)^{n}
  \sin(h_n(a)) \cdot \partial_a^j h_n(a).
\end{align}
Now to estimate the order of $\partial_a^j B^*_n(a)$, we have the following claim
whose proof we omit as it is analogous to that of Claim~\ref{claim_M}.
\begin{claim} \label{claim1}
\begin{equation}  \label{covergence}
B^*_n=O(1) \quad \text{and} \quad \partial_a^j B^*_n=o(1) \ \
\text{for all $1\leq |j|\leq d$}
\end{equation}
\end{claim}
Since the norms of the functions $B^*_n(a),\ \partial_a^j (\log
\lambda_a \gamma_a)$ are bounded from above and using that $\partial_a^j B^*_n=o(1)$ due to Claim~\ref{claim1} we get from~\eqref{eq:O1} that
\begin{align}
(\lambda_a \gamma_a)^{n}
  \sin(h_n(a)) \cdot \partial_a^j h_n(a)=O(n)+o(1).
\end{align}
Since $\sin(h_n(a))$ is bounded away from zero by the choice of
$\varphi_n$, resolving  the previous equation for $\partial_a^j
h_n(a)$
we obtain that $\partial_a^j h_n(a)=O(n(\lambda_a\gamma_a)^{-n})$.
This implies that
$\partial_a^j\varphi_n^*(a)=O((\lambda_a\gamma_a)^{-n})$ and
proves the formula for $|j|=1$. To prove the expression for any
index $\ell$, we will proceed by induction assuming that
$\partial_a^j h_n(a)= O(n^{|j|}(\gamma_a\lambda_a)^{-n})$ for any
$j$ with $1\leq |j| <\left|\ell\right|$ and will show that
$\partial_a^\ell h_n(a)=
O(n^{\left|\ell\right|}(\gamma_a\lambda_a)^{-n})$. This will imply
that
$\partial_a^\ell\varphi_n^*(a)=O(n^{\left|\ell\right|-1}(\lambda_a\gamma_a)^{-n})$,
as it is required to complete the proof of the lemma.

Take $\ell$ with $\left|\ell\right|> 1$, and fix $\iota<\ell$ such
that $\left|\iota\right|=1$. Since
$\partial_a^{\ell}B^*_n(a)=\partial_a^{\ell-\iota}(\partial_a^{\iota}B^*_n(a))$
we get from~\eqref{B*}, \eqref{eq:O1} and by using the Leibniz
formula,
\begin{equation} \label{eq:final}
 \partial_a^{\ell}B^*_n(a) = n\cdot
\partial_a^{\ell-\iota}(B^*_n(a)\cdot \partial_a(\log \lambda_a\gamma_a)) +
  \sum_{j:j\leq (\ell-\iota)} \binom{\ell-\iota}{j}\,
 \partial_a^{(\ell-\iota)-j}((\lambda_a\gamma_a)^{n}\sin(h_n(a)))\cdot
 \partial_a^{j+\iota} h_n(a).
\end{equation}
Now we will determine the order of the terms in the above equation.
\begin{claim} \label{claim2}
\begin{equation}
\partial_a^j
((\lambda_a\gamma_a)^n\sin(h_n(a)))=O(n^{|j|}(\lambda_a\gamma_a)^n)
\quad \text{for all $1\leq |j|\leq d$}
\end{equation}
\end{claim}
\begin{proof}[Proof of Claim~\ref{claim2}]
Observe from~\eqref{B*} that
$[(\lambda_a\gamma_a)^n\sin(h_n(a))]^2=(\lambda_a\gamma_a)^{2n}-(B^*_n)^2$.
Since   $\partial_a^j B^*_n=o(1)$ from Claim~\ref{claim1}, we get
that
$$\partial_a^j
((\lambda_a\gamma_a)^n\sin(h_n(a)))=O(\partial_a^j(\lambda_a\gamma_a)^n).$$
On the other hand it can be seen by an inductive argument that
$\partial_a^j(\lambda_a\gamma_a)^n=O(n^{|j|}(\lambda_a\gamma_a)^n)$,
which gives the required estimate.
\end{proof}

For $j<(\ell-\iota)$, by the induction hypothesis
$\partial_a^{j+\iota} h_n(a)=
O(n^{\left|j+\iota\right|}(\gamma_a\lambda_a)^{-n})$. Combining
this with Claim~\ref{claim2} we obtain that for $j<(\ell-\iota)$
$$\partial_a^{(\ell-\iota)-j}((\lambda_a\gamma_a)^{n}\sin(h_n(a)))\cdot
 \partial_a^{j+\iota} h_n(a)= O(n^{\left|(\ell-\iota)-j\right|}(\lambda_a\gamma_a)^n) \cdot O(n^{\left|j+\iota\right|}(\gamma_a\lambda_a)^{-n})=O(n^{\ell}).$$
On the other hand $\partial_a^{\ell}B^*_n(a) = o(1)$ and $n\cdot
\partial_a^{\ell-\iota}(B^*_n(a)\cdot \partial_a(\log \lambda_a\gamma_a))=O(n)$. Thus, putting all these estimates together in~\eqref{eq:final} and isolating the term corresponding with the index $j=\ell-\iota$ we obtain
$$(\lambda_a\gamma_a)^{n}\sin(h_n(a)))\cdot
 \partial_a^{\ell} h_n(a)=O(n)+O(n^{\left|\ell\right|})+o(1).$$
This implies that $\partial_a^\ell h_n(a)=
O(n^{\left|\ell\right|}(\gamma_a\lambda_a)^{-n})$ concluding the
proof.
\end{proof}

\bibliographystyle{alpha2}
\bibliography{bib}

\end{document}